\documentclass[11pt,
]{amsart}
\usepackage{a4wide}
\usepackage{verbatim,upref,amsxtra,amssymb,amscd,graphicx}

\usepackage{mathtools}

\usepackage[mathscr]{eucal}

\usepackage[english]{babel}

\usepackage{url}

\usepackage[
colorlinks=true, linkcolor=blue, citecolor=blue, urlcolor=blue]{hyperref}

\usepackage{enumitem}
\numberwithin{equation}{section}

\newtheorem{theorem}{Theorem}[section]
\newtheorem{proposition}[theorem]{Proposition}
\newtheorem{lemma}[theorem]{Lemma}
\newtheorem{corollary}[theorem]{Corollary}

\theoremstyle{definition}

\newtheorem{definition}[theorem]{Definition}
\newtheorem{example}[theorem]{Example}

\theoremstyle{remark}

\newtheorem{remark}[theorem]{Remark}

\DeclareMathOperator{\N}{\mathbb{N}}
\DeclareMathOperator{\Z}{\mathbb{Z}}

\DeclareMathOperator{\R}{\mathbb{R}}
\DeclareMathOperator{\C}{\mathbb{C}}
\DeclareMathOperator{\T}{\mathbb{T}}
\DeclareMathOperator{\h}{\mathbb{H}}
\DeclareMathOperator{\s}{\mathbb{S}} 
\DeclareMathOperator{\lh}{\ell ^1(\mathbb{H},\mathbb{C})}

\begin{document}
\allowdisplaybreaks\frenchspacing

\baselineskip=1.2\baselineskip

\title[A Wiener Lemma for the discrete Heisenberg group]{A Wiener Lemma for \\the discrete Heisenberg group:\\ 
\Small{Invertibility criteria and applications to algebraic dynamics}}

\author{Martin G\"oll} 
\address{Martin G\"oll: Mathematical Institute, University of Leiden, PO Box 9512, 2300 RA Leiden, The Netherlands} 

\author{Klaus Schmidt}
\address{Klaus Schmidt: Mathematics Institute, University of Vienna, Oskar-Morgenstern-Platz 1, A-1090 Vienna, Austria \newline 
\textup{and} \newline \indent
Erwin Schr\"odinger Institute for Mathematical Physics, Boltzmanngasse~9, A-1090 Vienna, Austria} \email{klaus.schmidt@univie.ac.at}

\author{Evgeny Verbitskiy}
\address{Evgeny Verbitskiy: Mathematical Institute, University of Leiden, PO Box 9512, 2300 RA Leiden, The Netherlands \newline
 \textup{and} \newline \indent
  Johann Bernoulli Institute for Mathematics and Computer Science, University of Groningen, PO Box 407, 9700 AK, Groningen, The Netherlands} \email{e.a.verbitskiy@rug.nl}
  


\thanks{
The authors would like to thank Karlheinz Gr\"ochenig, A.J.E.M. Janssen,  Hanfeng Li and Doug Lind for helpful discussions and insights.
Moreover, we thank Mike Keane for making us aware of the article \cite{2Brown} by Ian Brown. 
\newline \indent MG gratefully acknowledges support by a
Huygens Fellowship from Leiden University. 
\newline\indent
MG and EV would like to thank the Erwin Schr\"{o}dinger
Institute, Vienna, and KS the University of Leiden, for hospitality and support while some of this work was done.
}

\maketitle

\begin{abstract}
This article contains a Wiener Lemma for the convolution algebra $\ell^1(\h,\C)$ and group $C^\ast$-algebra $C^\ast(\h)$ of the discrete Heisenberg group $\h$.

At first, a short review of Wiener's Lemma in its classical form and general results about invertibility in group algebras of nilpotent groups will be presented.
The known  literature on this topic suggests that  invertibility investigations  in the group algebras of $\h$ rely on the complete  knowledge of $\widehat{\h}$ --  the dual of $\h$, i.e.,  the space of unitary equivalence classes of  irreducible unitary representations.
We will describe the dual of ${\h}$ explicitly and discuss its structure.

 Wiener's Lemma provides a convenient condition to verify invertibility in $\ell^1(\h,\C)$ and $C^\ast(\h)$  which bypasses $\widehat{\h}$.
The proof of Wiener's Lemma for $\h$ relies on local principles and can be generalised to countable nilpotent groups. As our analysis shows, the main representation theoretical objects to study invertibility in group algebras of nilpotent groups are the corresponding primitive ideal spaces.
Wiener's Lemma for $\h$ has interesting applications in algebraic dynamics and Time-Frequency Analysis which will be presented  in this article as well.

\end{abstract}

\section{Motivation}
\label{2intro}
Let $\Gamma$ be a countably infinite discrete group.
The aim of this article is to find a verifiable criterion -- a Wiener Lemma -- for invertibility in the group algebra
\begin{equation*}
\ell^1(\Gamma,\C) \coloneqq \biggl \{ (f_\gamma)_{\gamma\in\Gamma}\,:\, \sum_{\gamma\in\Gamma} |f_\gamma | < \infty \biggr \} \,,
\end{equation*}
in particular for the case where $\Gamma$ is the discrete Heisenberg group $\h$.

Our main motivation to study this problem is an application in the field of algebraic dynamics which we introduce first.
An \textit{algebraic $\Gamma$-action} is a homomorphism $\alpha\colon \Gamma \longrightarrow \text{Aut} (X)$ from $\Gamma$ to the group of
automorphisms of a compact metrisable abelian group $X$ \cite{2KS}.

We are especially interested in \textit{principal actions} which are defined as follows.
Let $f$ be an element in the \textit{integer group ring} $\Z[\Gamma]$, i.e., the ring of functions $\Gamma\longrightarrow \Z$ with finite support.
The Pontryagin dual of the discrete abelian group $\Z[\Gamma]/\Z[\Gamma ]f$ will be denoted by $X_f \subseteq \T^{\Gamma}$, where $\T=\mathbb{R}/\mathbb{Z}$ (which will be identified with the unit interval $(0,1]$). Pontryagin's duality theory of locally compact abelian groups tells us that $X_f$ can be identified with the annihilator of the principal left ideal $\Z[\Gamma]f$, i.e.,
\begin{equation}\label{2d:principal}
	X_f = (\Z[\Gamma]f)^\perp =\biggl \{x\in\T^{\Gamma}\,:\, \sum_{\gamma\in\Gamma} f_\gamma x_{\gamma'\gamma} =0 \,\,\text{for every}\,\gamma'\in\Gamma \biggr \}\,.
\end{equation}
The \textit{left shift-action}  $\lambda$ on $\T^\Gamma$ is defined by $(\lambda^\gamma x)_{\gamma '} = x_{\gamma^{-1}\gamma'}$ for every
$x \in \T ^\Gamma$ and $\gamma,\gamma ' \in \Gamma$.
Denote by $\alpha_f$ the restriction of $\lambda$ on $\T^{\Gamma}$ to $X_f$. The pair $(X_f,\alpha_f)$ forms an algebraic dynamical system which we call \textit{principal $\Gamma$-action} -- because it is defined by a principal ideal (cf. \eqref{2d:principal}).

Since a principal $\Gamma$-action $(X_f,\alpha_f)$ is completely determined by an element $f\in \mathbb{Z}[ \Gamma ]$, one should be able to express its dynamical properties in terms of properties of $f$.
Expansiveness is such a dynamical property which allows a nice algebraic interpretation.
Let $(X,\alpha)$ be an algebraic dynamical system and $d$ a translation invariant metric on $X$. The $\Gamma$-action $\alpha$ is \textit{expansive} if  there exists a constant $\varepsilon >0$ such that
\begin{equation*}
 \sup_{\gamma\in\Gamma} d(\alpha^{\gamma} x, \alpha^{\gamma} y) > \varepsilon \,,
\end{equation*}
for all pairs of distinct elements $x,y\in X$.
We know from \cite[Theorem 3.2]{2DS} that $(X_f,\alpha_f)$ is expansive if and only if $f$ is invertible in $\ell^1(\Gamma,\R)$.
This result was proved already in the special cases $\Gamma=\Z^d$ and for groups $\Gamma$ which are nilpotent in \cite{2KS} and in \cite{2ER}, respectively.
Although, this result is a complete characterisation of expansiveness, it is in general  hard  to check whether $f$ is invertible in $\ell^1(\Gamma,\R)$ or not.

\subsection{Outline of the article}
In Section \ref{2section2} we will recall known criteria for invertibility in symmetric unital Banach algebras $\mathcal A$.
The most important result links invertibility investigations in $\mathcal A$ to the representation theory of $\mathcal A$.
More precisely, the existence of an inverse $a^{-1}$ of $a \in\mathcal A$ is equivalent to the invertibility of the operators $\pi(a)$   for every irreducible unitary representation $\pi$  of $\mathcal A$.
The representation theory of $\h$ is unmanageable as we will demonstrate in Section \ref{2section3}.

Theorem \ref{2t:main} -- Wiener's Lemma for the discrete Heisenberg group -- is the main result of this paper and allows one to restrict the attention to certain `nice' and canonical irreducible representations for questions concerning invertibility in the group algebra of the discrete Heisenberg group $\h$.
The proof of Theorem \ref{2t:main} can be found in Section \ref{2section4}. Moreover, as will be shown in Section \ref{2section4} as well,
invertibility of $f\in\Z[\h]$ in $\ell^1(\h,\R)$ can be verified with the help of the finite-dimensional irreducible unitary representations of $\h$.

In Section \ref{2section5}
we generalise Theorem \ref{2t:main} to countable discrete nilpotent groups $\Gamma$. This result says that an element $a$ in $C^\ast(\Gamma)$ is invertible if and only if for every primitive ideal $\mathtt I$ of $C^\ast(\Gamma)$
  the projection of $a$ onto the quotient space
$C^\ast(\Gamma)/\mathtt I$ is invertible. As we will see, the primitive ideal space is more accessible than the space of irreducible representations and easy to determine. Moreover, this Wiener Lemma for nilpotent groups can be converted to a statement about invertibility of evaluations of irreducible monomial representations.

 In Section \ref{2section6} we will explore a connection to Time-Frequency Analysis. Allan's local principle (cf. Section \ref{2section4}) directly links localisations of $\lh$ to twisted convolution algebras and hence,  the representations of $\h$
 and the relevant representation theory in the field of Time-Frequency Analysis coincide. In order to highlight this connection even more,
 Time-Frequency Analysis might be interpreted as the Fourier theory on the discrete Heisenberg group $\h$; due to the striking similarities to the Fourier Analysis of the additive group $\Z$ and its group algebras.
  Moreover, we give an alternative proof of Wiener's Lemma for  twisted convolution algebras, which  only uses  the representation theory of $\h$.  Theorem \ref{2stonevneumannspectrum} -- which is based on a result of Linnell (cf. \cite{2L}) -- gives a full description of the spectrum of the operators $\pi (f)$ acting on $L^2(\R,\C)$, where $\pi$ is a Stone-von Neumann representation (cf. (\ref{2def:stonevneu}) for a definition) and $f\in\Z[\h]$.

Section \ref{2section7} contains applications of  Theorem \ref{2t:main} and Wiener's Lemma for twisted convolution algebras, in particular,
 conditions for non-invertibility for `linear' elements in $f\in\Z[\h]$.

\section{Invertibility in group algebras and Wiener's Lemma: A review}\label{2section2}

In this section we review known conditions for invertibility in group algebras of nilpotent groups $\Gamma$.
First of all we refer to the article \cite{2Gro} by K. Gr\"ochenig for a modern survey of Wiener's Lemma and its variations. Gr\"ochenig's survey focuses on two main topics, namely on invertibility of convolution operators on $\ell^p$-spaces (cf. Subsection \ref{subs:sym} and in particular Theorem \ref{2t:barnes}) and inverse-closedness. Moreover, Gr\"ochenig  explains how these topics are related to questions on invertibility in Time-Frequency analysis and invertibility in group algebras.
Although, Wiener's Lemma for convolution operators is stated here as well it will play an insignificant role in the rest of the paper. However, we would like to bring the reader's attention to Theorem \ref{2t:den-schmidt} which is yet another result which relates invertibility in $\ell^1(\Gamma,\C)$ to invertibility of convolution operators. This result is \emph{completely} independent of Theorem \ref{2t:barnes} and holds in much greater generality.

In this review we will explain why a detailed understanding of the space of irreducible representations of a nilpotent group $\Gamma$ is of importance for invertibility investigations in the group algebras of $\Gamma$. Furthermore, we will present Gelfand's results on invertibility in commutative Banach algebras in the form of local principles; which will be discussed in greater detail in later sections of this article.

We start the discussion with Wiener's Lemma in its classical form.
Let us denote by $\mathcal A (\T)$ the Banach algebra of functions with absolutely convergent Fourier series on $\T$.
	\begin{theorem}[Wiener's Lemma]
	An element $F\in \mathcal A (\T)$ is invertible, i.e. $1/F \in \mathcal A (\T)$, if and only if $F(s)\not = 0$ for all $s \in \T$.
	\end{theorem}

Before we start our review of more general results let us mention the concept of \textit{inverse-closedness} which originates from Wiener's Lemma as well. The convolution algebra
$\ell^1(\Z,\C)$ is isomorphic to $\mathcal A (\T)$ and hence $\ell^1(\Z,\C)$ can be embedded in the larger Banach algebra of continuous functions $C(\T,\C)$ in a natural way. The fact that
$F\in\mathcal A (\T)$ is invertible in $\mathcal A (\T)$ if and only if $F$ is invertible in  $C(\T,\C)$ leads to the question:
\textit{for which pairs of nested unital Banach algebras $\mathcal A ,\mathcal B$ with $\mathcal A \subseteq \mathcal B$ and with the same multiplicative identity element does the following implication hold:}
	\begin{equation}\label{2wienerpair}
	a\in \mathcal A \quad\text{and}\quad a^{-1}\in \mathcal B \,\,\Longrightarrow \,\, a^{-1}\in \mathcal A\,.
	\end{equation}
In the literature a pair of Banach algebras which fulfils \eqref{2wienerpair} is called a \textit{Wiener pair}.

Wiener's Lemma was the starting point of Gelfand's study of invertibility in commutative Banach algebras.
Gelfand's theory links the question of invertibility in a commutative Banach algebra $\mathcal A$ to the study of its irreducible representations and the compact space of maximal ideals $\textup{Max} (\mathcal A)$. We collect in the following theorem several criteria for invertibility in unital commutative Banach algebras.

\begin{theorem}[cf. \cite{2Folland}] \label{2Gelfand}
Suppose $\mathcal A$ is a unital commutative Banach algebra. The set of irreducible representations of $\mathcal A$ is  isomorphic to the compact space of maximal ideals $\textup{Max} (\mathcal A)$. Furthermore, the following statements are equivalent
\begin{enumerate}
	\item $a \in \mathcal A$ is invertible;
	\item $a \not \in m$ for all $m \in \textup{Max} (\mathcal A)$;
	\item $\Phi_m(a)$ is invertible in $\mathcal A /m$ for all $m \in \textup{Max} (\mathcal A)$, where $\Phi_m\colon \mathcal A \longrightarrow \mathcal A /m \cong \C$ is the canonical projection map;
	\item $\Phi_m(a) \not = 0$ for all $m \in \textup{Max} (\mathcal A)$;
	\item $\pi(a)v \not = 0$ for every one-dimensional irreducible unitary representation $\pi$ of $\mathcal A$ and $v\in \C \smallsetminus \{ 0\}$ (definitions can be found in Subsection \ref{subs:reptheo}).
\end{enumerate}
\end{theorem}
The main goal of this article is to prove that similar results hold for group algebras of nilpotent groups and, in particular, for the discrete Heisenberg group.

In this article we concentrate on the harmonic analysis of rings associated with a countably infinite group $\Gamma$ furnished with the discrete topology.
Beside $\Z[\Gamma]$ and $\ell^1(\Gamma,\C)$ we are interested in $C^\ast(\Gamma)$, the \textit{group-$C^\ast$-algebra} of $\Gamma$, i.e., the enveloping $C^\ast$-algebra of $\ell^1(\Gamma,\C)$.

Let $\ell ^\infty (\Gamma ,\mathbb{C}) $ be the space of bounded complex-valued maps.
We write a typical element $f \in \ell^\infty(\Gamma,\C)$ as a formal sum $\sum_{\gamma\in\Gamma} f_\gamma \cdot \gamma $, where $f_\gamma= f(\gamma)$. The \textit{involution} $f\mapsto f^\ast$ is defined by $f^\ast=\sum_{\gamma\in\Gamma} \bar{f}_{\gamma^{-1}} \cdot \gamma $.
The product of $f\in \ell^1(\Gamma,\C)$ and $g\in \ell^\infty(\Gamma,\C)$ is given by \textit{convolution}
	\begin{equation}
	\label{2eq:convolution}
fg=\sum_{\gamma ,\gamma '\in\Gamma }f_\gamma g_{\gamma '}\cdot \gamma \gamma ' =\sum_{\gamma ,\gamma '\in\Gamma }f_\gamma g_{\gamma ^{-1}\gamma '}\cdot  \gamma '\,.
	\end{equation}
For $1\le p <\infty $ we set
$$\ell ^p(\Gamma ,\mathbb{C})=\biggl \{f=(f_\gamma )\in\ell ^\infty (\Gamma ,\mathbb{C})
\,\colon \,\|f\|_p=\bigl(\sum_{\gamma \in\Gamma }|f_\gamma |^p\bigr)^{1/p}<\infty \biggr \}\,.$$

\subsection{Representation theory} \label{subs:reptheo}
We recall at this point some relevant definitions and results from representation theory, which will be used later. Moreover, we will state results for symmetric Banach-$^\ast$-algebras which are in the spirit of Wiener's Lemma.
\subsubsection*{Unitary Representations}
Let $\mathcal{H}$ be a complex Hilbert space with inner product $\langle \cdot ,\cdot \rangle $. We denote by $\mathcal{B}(\mathcal{H})$ the algebra of bounded linear operators on $\mathcal{H}$, furnished with the strong operator topology. Further, denote by $\mathcal{U}(\mathcal{H})\subset \mathcal{B}(\mathcal{H})$  the group of unitary operators on $\mathcal{H}$. If $\Gamma $ is a countable group, a \textit{unitary representation} $\pi $ of $\Gamma $ is a homomorphism $\gamma \mapsto \pi (\gamma )$ from $\Gamma $ into $\mathcal{U}(\mathcal{H})$ for some complex Hilbert space $\mathcal{H}$. Every unitary representation $\pi $ of $\Gamma $ extends to a $^\ast$-representation of $\ell ^1(\Gamma ,\mathbb{C})$, which is again denoted by $\pi $, and which is given by the formula $\pi (f)=\sum_{\gamma \in \Gamma }f_\gamma \pi (\gamma )$ for $f=\sum_{\gamma \in \Gamma }f_\gamma \cdot \gamma \in \ell ^1(\Gamma ,\mathbb{C})$. Clearly, $\pi (f^\ast)=\pi (f)^\ast$. The following theorem was probably first published in \cite{2GN} but we refer to \cite[Theorem 12.4.1]{2Palmer}.

\begin{theorem} \label{2t : reps}
Let $\Gamma$ be a discrete group. Then there are bijections between
\begin{itemize}
\item the class of unitary representations of $\Gamma$;
\item the class of non-degenerate\footnote{A representation $\pi$ of a Banach $^\ast$-algebra $\mathcal A$ is called \textit{non-degenerate} if there is no non-zero vector $v\in\mathcal H_\pi$ such that $\pi(a)v=0$ for every $a\in\mathcal A$. } $^\ast$-representations of $\ell^1(\Gamma,\C)$;
\item the class of non-degenerate $^\ast$-representations of $C^\ast(\Gamma)$.
\end{itemize}
Moreover, these bijections respect unitary equivalence and irreducibility.
\end{theorem}
Hence the representation theories of $\Gamma$, $\ell^1(\Gamma,\C)$ and $C^\ast(\Gamma)$ coincide. In consideration of this result we will use the same symbol for a unitary representation of $\Gamma$ and its corresponding $^\ast$-representations of the group algebras $\ell^1(\Gamma,\C)$ and $C^\ast(\Gamma)$.

\subsubsection*{States and the GNS construction}
Suppose that $\mathcal A$ is a unital $C^\ast$-algebra.
A positive linear functional $\phi\colon \mathcal A \longrightarrow \C$ is a \textit{state} if $\phi(1_{\mathcal A})=1$.
We denote  by $\mathcal S (\mathcal A )$ the space of states of $\mathcal A$, which is a weak$^\ast$-compact convex subset of the dual space of $\mathcal A$.
The extreme points of $\mathcal S (\mathcal A)$ are called \textit{pure states}.

A representation $\pi$ of $\mathcal A$ is \textit{cyclic} if there exists a vector $v \in \mathcal H _\pi$ such that the set $\{ \pi(a)v \,:\,a\in\mathcal A\}$
is dense in $\mathcal H _\pi$, in which case $v$ is called a cyclic vector.
The Gelfand-Naimark-Segal (GNS) construction links the cyclic representations of $\mathcal A$ and the states of $\mathcal A$ in the following way.
If $\pi$ is a cyclic representation with a cyclic unit vector $v$, then $\phi_{\pi,v}$, defined by
\begin{equation*}
\phi_{\pi,v}(a)= \langle \pi(a)v,v\rangle
\end{equation*}
for every  $a \in \mathcal A$, is a state of $\mathcal A$. If $\pi$ is irreducible, then $\phi_{\pi,v}$ is a pure state.
Moreover, for every state $\phi$ of $\mathcal A$ there is a cyclic representation $(\pi_\phi,\mathcal H _\phi )$ and a  cyclic unit vector $v_\phi \in \mathcal H _\phi$ such that $\phi(a)=\langle \pi_\phi(a)v_\phi,v_\phi \rangle$ for every $a\in \mathcal A$.
The pure states of $\mathcal A$ correspond to irreducible representations of $\mathcal A$ (up to unitary equivalence) via the GNS construction.

\subsubsection*{Type I groups}
Let $\mathcal H$ be a Hilbert space.
The \textit{commutant} of a subset $N$  of $\mathcal B (\mathcal H )$ is the set
	\begin{equation*}
	N^{'}\coloneqq \{\textup{A} \in \mathcal B (\mathcal H)\,:\, \textup{A}\textup{S}=\textup{S}\textup{A} \,\,\text{for all}\,\, \textup{S} \in N\}\,.
	\end{equation*}
A \textit{von Neumann algebra} $\mathcal N$ is a $^\ast$-subalgebra of bounded operators on some Hilbert space $\mathcal H$ which fulfils $\mathcal N = (\mathcal N^{'})^{'}$.
The von Neumann algebra $\mathcal N _\pi$ generated by a unitary representation $\pi$ of a group $\Gamma$,
is the smallest von Neumann algebra which contains $\pi (\Gamma)$.

We call a representation $\pi$ a \textit{factor} if $\mathcal N _\pi \cap \mathcal N^{'} _\pi = \C \cdot 1_{\mathcal B (\mathcal H _\pi)}$. A group is of \textit{Type I} if every factor representation is a direct sum of copies of an irreducible representation.

\subsubsection*{Induced and monomial representations}

Let $H$ be a subgroup of a countably infinite group $\Gamma$. Suppose $\sigma$ is a unitary representation of $H$ with representation space $\mathcal H _\sigma$.
A natural way to extend  the representation $\sigma$ of $H$ to a representation of $\Gamma$ is as follows:
consider the Hilbert space $\mathcal H_\sigma^{\Gamma}$ consisting of all maps  $F\in L^2 (\Gamma ,\mathcal H _\sigma)$  which satisfy
	\begin{equation*}\label{restricted}
	F( \gamma \delta)= \sigma (\delta)F(\gamma)\quad\text{for every  $\delta\in H$ and $\gamma \in\Gamma$.}
	\end{equation*}
The \textit{induced representation}
$\text{Ind}^\Gamma _H(\sigma) \colon \Gamma\ni\gamma\mapsto \text{Ind}^\Gamma_H(\sigma)(\gamma)\in \mathcal B( \mathcal H_\sigma^{\Gamma})$  is then defined  by
$$
\text{Ind}^\Gamma_H(\sigma)(\gamma) F(\gamma') = F(\gamma'\gamma )
\quad\forall \gamma'\in \Gamma.
$$
Hence, $\text{Ind}_H^\Gamma(\sigma)$ can be viewed as  the
 right regular representation of $\Gamma$ acting on the Hilbert space $\mathcal H_\sigma^\Gamma$.
 
This construction will become more transparent when we discuss specific examples
below.

A representation of $\Gamma$ is called \textit{monomial} if it is unitarily equivalent to a representation induced from a one-dimensional representation of a subgroup of $\Gamma$.
	\begin{theorem}[\cite{2Hannabuss}] \label{2typeI/mono}
	If $\Gamma$ is a nilpotent group of Type I, then all its irreducible representations are monomial.
	\end{theorem}

\subsection{Symmetric Banach-$^\ast$-algebras} \label{subs:sym}

Let $\mathcal A$ be a Banach algebra with multiplicative identity element $1_{\mathcal A}$. The \textit{spectrum} of $a\in \mathcal A$ is the set of elements
$c\in \C$ such that $a-c 1_{\mathcal A}$ is not invertible in $\mathcal A$ and will be denoted by $\sigma (a)$.

In order to study invertibility in $\ell^1(\Gamma,\C)$ and $C^\ast(\Gamma)$ in the non-abelian setting we will try to find criteria similar to those described
 in Theorem \ref{2Gelfand}. For this purpose the following definition will play a key role.
\begin{definition}
A unital Banach-$^\ast$-algebra $\mathcal A$ is \textit{symmetric} if for every element $a\in\mathcal A$ the spectrum of $a^\ast a$ is non-negative, i.e., $\sigma(a^\ast a) \subseteq [0,\infty)$.
\end{definition}
Typical examples of symmetric Banach-*-algebras are $C^\ast$-algebras.

We turn to the study of nilpotent groups and their associated group algebras.

	\begin{theorem}[\cite{2Hulanicki}]
	Let $\Gamma$ be a countably infinite discrete nilpotent group. Then the Banach-$^\ast$-algebra $\ell^1(\Gamma,\C)$ is symmetric.
	\end{theorem}

The reason why it is convenient to restrict to the study of invertibility in symmetric unital Banach-$^\ast$-algebra is demonstrated by the following theorems,
which show similarities to Wiener's Lemma and Theorem \ref{2Gelfand}, respectively.

{For the class of symmetric group algebras one has the following important result on  inverse-closedness.}

	\begin{theorem}[{\cite{2Lud}, see also \cite[Theorem 11.4.1 and Corollary 12.4.5]{2Palmer}}]\label{2t:Wienerpair}
	If $\ell^1(\Gamma,\C)$ is a symmetric Banach-$^\ast$-algebra, then
		\begin{enumerate}
		\item $\ell^1(\Gamma,\C)$ is semisimple, i.e., the intersection of the kernels of all the irreducible representations of
		$\ell^1(\Gamma,\C)$ is trivial.
		\item $\ell^1(\Gamma,\C)$ and its enveloping $C^\ast$-algebra $C^\ast(\Gamma)$ form a Wiener pair.
		\end{enumerate}
	\end{theorem}

Next we are discussing spectral invariance of convolution operators. It is a well known fact (cf. \cite{2Gro}) that invertibility of $f\in\ell^1(\Z,\C)$ can be validated by studying invertibility of the convolution operator $\textup{C}_f$ acting on the Hilbert space $\ell^2(\Z,\C)$. Moreover, the spectrum of $\textup{C}_f$
is independent of the domain, i.e., the spectrum of the operator $\textup{C}_f\colon\ell^p(\Z,\C)\longrightarrow \ell^p(\Z,\C)$ is the same for all $p\in [1,\infty]$. As the following theorem shows, this result is true for a large class of groups, in particular, for all finitely generated nilpotent groups.
	\begin{theorem}[\cite{2Barnes}]\label{2t:barnes}
	Let $f\in \ell^1(\Gamma,\C)$ and  $\textup{C}_f$ the associated convolution operator on $\ell^p(\Gamma,\C)$. For all $1\leq p \leq \infty$
	one has
	$\sigma_{\mathcal B (\ell^p(\Gamma,\C))}(\textup{C}_f)=\sigma_{\mathcal B (\ell^2(\Gamma,\C))}(\textup{C}_f)$ if and only if $\Gamma$ is amenable and $\ell^1(\Gamma,\C)$ is a symmetric
	Banach-$^\ast$-algebra.
	\end{theorem}
In particular, for a nilpotent group $\Gamma$, $f\in\ell^1(\Gamma,\C)$ is invertible in $\ell^1(\Gamma,\C)$ if and only if $0 \notin \sigma_{\mathcal B (\ell^p(\Gamma,\C))}(\textup{C}_f)$ for any $p\in [1,\infty]$.

Let us now give a condition for invertibility of an element $\ell^1(\Gamma,\C)$, where $\Gamma$ is an arbitrary discrete countably infinite group, in terms of the point spectrum of the corresponding convolution operator.
	\begin{theorem}[{\cite[Theorem 3.2]{2DS}}]\label{2t:den-schmidt}
	An element $f\in \ell^1(\Gamma,\C)$ is invertible in $\ell^1(\Gamma,\C)$ if and only if
	\begin{equation*}
	K_\infty(f)\coloneqq \{ g\in \ell^\infty(\Gamma,\C)\,:\, \textup{C}_f g=0 \}=\{0\}\,.
	\end{equation*}
	\end{theorem}

	This theorem says that it is enough to check if $0$ is an eigenvalue of the left convolution operator $\textup{C}_f\colon \ell^\infty(\Gamma,\C)\longrightarrow \ell^\infty(\Gamma,\C)$ in order to determine whether $f$ is invertible or not (cf. \eqref{2eq:convolution}).

Finally, we present a condition for invertibility in a symmetric unital Banach-$^\ast$-algebra $\mathcal A$ which links invertibility in $\mathcal A$  to its representation theory.
\begin{theorem}[\cite{2Neumark}] \label{2t:neumark}
	An element $a$ in a symmetric unital Banach-$^\ast$-algebra $\mathcal A$ is not left invertible in $\mathcal A$ if and only
	if there exists a pure state $\phi$ with $\phi(a^\ast a)=0$. Equivalently, $a$ is not left invertible if and only if there exists
	an irreducible representation $\pi$ of $\mathcal A$ and a unit vector $u \in \mathcal H _\pi$ such that $\pi(a)u=0$.
	\end{theorem}

This result should be compared with Gelfand's theory for commutative Banach algebras. Wiener's Lemma for $\ell^1(\Z,\C)$ says that an element
$f \in \ell^1(\Z,\C)$ is invertible if and only if the Fourier-transform of $f$ does not vanish on $\mathbb{T}$, i.e., $(\mathcal F f) (s) \not = 0$ for all $s\in \mathbb{T}$.\footnote{\,To fix notation: for $F\in L^2(\mathbb{T},\lambda _\mathbb{T})$ (where $\lambda _\mathbb{T}$ is the Lebesgue measure on $\mathbb{T}$), the Fourier transform $\hat{F}\colon \mathbb{Z}\longrightarrow \mathbb{C}$ is defined by $\hat{F}_n=\int _\mathbb{T}F(s)e^{-2\pi ins}\, d\lambda _\mathbb{T}(s)$. The Fourier transform $(\mathcal F g)\colon \mathbb{T}\longrightarrow \mathbb{C}$ of $g\in \ell ^2(\mathbb{Z},\mathbb{C})$ is defined by $(\mathcal F g)(s)=\sum_{n\in \mathbb{Z}}g_n e^{2\pi ins}$.} The Fourier-transform of $f$, evaluated at the point $\theta\in \mathbb{T}$, can be viewed as the evaluation of the one-dimensional irreducible unitary representation $\pi_\theta\colon n\mapsto e^{2\pi in\theta}$ of $\Z$ at $f$, i.e.,
	\begin{equation*}
	(\mathcal F f) (\theta)= \Bigl(\sum\nolimits_{n\in\Z} f_n \pi_\theta (n)\Bigr)1  = \pi_\theta(f)1\,.
	\end{equation*}

We will explain in the next section that it is not feasible to describe explicitly the space of unitary equivalence classes of irreducible representations of a non-Type I group. Hence, Theorem \ref{2t:neumark} seems to be of limited use for investigating invertibility of an element $f \in \ell^1(\Gamma,\C)$ for a non-Type I nilpotent group $\Gamma$. However, as we will see later, it is one of the key results for obtaining a Wiener Lemma for $\ell^1(\Gamma,\C)$.

\section{The dual of the discrete Heisenberg group and a Wiener Lemma}\label{2section3}

In this section we explain how results from ergodic theory give insight into the space of irreducible representations of the discrete Heisenberg group, but that this space has no reasonable parametrisation and is therefore not useful for determining invertibility in the corresponding group algebras (cf. Theorem \ref{2t:neumark}). At the end of this section, we will state our main result -- a Wiener Lemma for the discrete Heisenberg group $\h$ -- which allows one to restrict the attention to certain canonical representations of $\h$ which \textit{can} be parametrised effectively and used for solving the invertibility problem.

\subsection{The dual of a discrete group}
Let $\Gamma$ be a countable discrete group.
Denote by $\widehat \Gamma$ the \textit{dual} of $\Gamma$, i.e., the set of all unitary equivalence classes of irreducible unitary representations of $\Gamma$.
\begin{definition}
Let $\mathcal A$ be a $C^\ast$-algebra. A closed two-sided ideal $\mathtt I$ of $\mathcal A$ is \textit{primitive} if  there exists an irreducible representation $\pi$ of $\mathcal A$ such that $\ker (\pi )=\mathtt I$. The set of primitive ideals of $\mathcal A$ is denoted by $\text{Prim}(\mathcal A)$.
\end{definition}

Suppose that the group $\Gamma$ is not of Type I. Then certain pathologies arise:
\begin{itemize}
\item The map $\widehat \Gamma \longrightarrow \text{Prim}(C^\ast(\Gamma))$ given by $\pi \mapsto \ker (\pi)$  is not injective.
In other words,
if $\pi_1, \pi_2 \in \widehat \Gamma$, then $\ker (\pi_1 )=\ker (\pi_2 )$ does not necessarily imply that $\pi_1$ and $\pi_2$ are unitarily equivalent.

\item $\widehat \Gamma$ is not  behaving nicely neither as a topological space nor as a measurable space in its natural topology or Borel structure, respectively (cf. \cite[Chapter 7]{2Folland} for an overview).
\end{itemize}

Furthermore, there are examples where the direct integral decomposition of a representation is not unique, in the sense that there are disjoint measures $\mu,\nu$ on $\widehat \Gamma$ such that
$
\int_{\widehat \Gamma}^\oplus \pi d\mu$ and
$
\int_{\widehat \Gamma}^\oplus \pi d\nu$ are unitarily equivalent.
 Moreover, we cannot assume that all irreducible
representations are induced from one-dimensional representations of finite-index subgroups, as is the case for nilpotent groups of Type I by Theorem \ref{2typeI/mono}.

\subsection{The discrete Heisenberg group and its dual}

The discrete Heisenberg group $\h$ is generated by $S=\{x,x^{-1},y,y^{-1}\}$, where
	\begin{equation*}
x = \left(
	\begin{matrix}
1 & 1 & 0
	\\
0 & 1 & 0
	\\
0 & 0 & 1
	\end{matrix}
\right), \quad y = \left(
	\begin{matrix}
1 & 0 & 0
	\\
0 & 1 & 1
	\\
0 & 0 & 1
	\end{matrix}
\right) .
	\end{equation*}
The centre of $\h$ is generated by
	\begin{equation*}
\smash[t]{z= xyx^{-1}y^{-1}= \left(
	\begin{matrix}
1 & 0 & 1
	\\
0 & 1 & 0
	\\
0 & 0 & 1
	\end{matrix}
\right).}
	\end{equation*}
The elements $x,y,z$ satisfy the following commutation relations
	\begin{equation}
	\label{2eq:relations}
xz=zx,\enspace yz=zy,\enspace x^ky^l=y^lx^kz^{kl},\enspace k,l\in\mathbb{Z}.
	\end{equation}
The discrete Heisenberg group is nilpotent and hence amenable.

Since $\h$ does not possess an abelian normal subgroup of finite index it is not a group of Type I (cf. \cite{2Thoma}), and hence the space of irreducible representations  does not have any nice structure as discussed above. As we will show below, one can construct uncountably many unitarily inequivalent  irreducible representations of $\h$ for every irrational $\theta \in\mathbb{T}$.  These representations arise from certain singular measures on $\T$.
This fact is well-known to specialists, but details
are not easily accessible in the literature.
Since these results are important for our understanding of invertibility, we present
this construction in some detail for the convenience of the reader.
We would like to mention first that Moran announced in \cite{2M} a construction of unitary representations of $\h$ using the same approach as presented here. These results were not published as far as we know. Moreover, Brown \cite{2Brown} gave examples of unitary irreducible representations  of the discrete Heisenberg group which are not monomial.

Let $(X,\mathfrak{B},\mu )$ be a measure space, where  $X$ is a compact metric space,
$\mathfrak{B}$ is a Borel $\sigma$-algebra, and $\mu$ a finite measure.

\begin{definition}
	A probability measure $\mu$ is \textit{quasi-invariant} with respect to a homeomorphism $\phi\colon X\longrightarrow X$ if $\mu(B)=0$ if and only if $\mu(\phi B)=0$, for $B \in \mathfrak{B}$.
	A quasi-invariant measure $\mu$ is ergodic  if
	\begin{equation*}
		 B \in \mathfrak{B}\quad \textup{and}\quad \phi B= B \implies \mu(B) \in \{0,1\} \,.
	\end{equation*}
\end{definition}
In \cite{2Keane} uncountably many inequivalent ergodic quasi-invariant measures for every irrational rotation of the circle were constructed.
Later it was shown in \cite{2KW} that a homeomorphism $\phi$ on a compact metric space $X$ has uncountably many inequivalent non-atomic ergodic quasi-invariant measures if and only if $\phi$ has a recurrent point $x$, i.e., $\phi^n (x)$ returns infinitely often to any punctured neighbourhood of $x$.

Let $\Z$ act on $\T$ via rotation
	\begin{equation}
	\label{2eq:Rtheta}
\textup{R}_\theta\colon t \mapsto t +\theta \mod 1
	\end{equation}
by an irrational angle $\theta\in\T$.

\begin{theorem}\label{t:representations}
For each irrational $\theta \in\T$ there is a bijection between the set of
ergodic $\textup{R}_\theta$-quasi-invariant probability measures on $\T$ and the set of irreducible representations $\pi$ of $\h$
with $\pi(z)=e^{2\pi i \theta}$.
\end{theorem}

We use the measures found in \cite{2Keane} to construct unitary irreducible representations of $\h$.
Suppose $\mu$ is an ergodic $\textup{R}_\theta$-quasi-invariant probability measure on $\T$.
Let $\textup{T}_{\theta ,\mu }\colon L^2(\T,\mu) \longrightarrow L^2(\T,\mu)$ be the unitary operator defined by
	\begin{equation}
	\label{2eq:T-theta-mu}
(\textup{T}_{\theta ,\mu }F)(t) = \sqrt{\frac{d\mu(t+\theta)}{d\mu(t)}} F(t+\theta) = \sqrt{\frac{d\mu(\textup{R}_{\theta} t)}{d\mu(t)}} F(\textup{R}_{\theta} t)\,,
	\end{equation}
for every $F \in L^2(\T,\mu)$ and $t \in \T$. The operator $\textup{T}_{\theta ,\mu }$ is well-defined because of the quasi-invariance of $\mu$. Consider also the unitary operator $\textup{M}_\mu $ defined by
	\begin{equation}
	\label{2eq:M-mu u}
(\textup{M}_\mu F)(t) = e^{2 \pi i t} F(t) \,,
	\end{equation}
for every $F \in L^2(\T,\mu)$ and $t \in \T$.

We will show that the representation $\pi _{\theta ,\mu }$ of $\h$ defined by
\begin{equation}\label{2eq:pimu}
\pi_{\theta ,\mu }(x)\coloneqq \textup{T}_{\theta ,\mu } \,,\quad \quad\pi_{\theta ,\mu }(y)\coloneqq \textup{M}_\mu \quad\text{and}\quad \pi_{\theta ,\mu }(z)\coloneqq e^{2\pi i \theta}
\end{equation}
is irreducible.
Obviously, $ \textup{T}_{\theta ,\mu } \textup{M}_\mu = e^{2 \pi i \theta} \textup{M}_\mu \textup{T}_{\theta ,\mu }  = \pi_{\theta ,\mu } (z) \textup{M}_\mu \textup{T}_{\theta ,\mu }$.
\begin{lemma}\label{2lemma:irreducible}
The unitary representation $\pi_{\theta ,\mu }$ of $\h$ given by \eqref{2eq:pimu} is irreducible.
\end{lemma}
\begin{proof}
Every element in $L^2(\T,\mu)$ can be approximated
by linear combinations of elements in the set
\begin{equation*}\{\textup{M}_\mu ^n1\,:\,n\in\Z\}=\{t \mapsto e^{2 \pi i n t}\,:\,n\in\Z \}\,.\end{equation*}
A bounded linear operator $\textup{O}$ on $L^2(\T,\mu)$, which commutes with all operators of the form $\textup{M}_\mu ^n$, $n \in\Z$, and hence with multiplication with any $L^\infty$-function, must be a multiplication operator, i.e.,
$\textup{O} F(t)= G(t)\cdot F(t)$ for some $G \in L^\infty(\T,\mu)$. Indeed, if $\textup{O}$ commutes with multiplication by $H \in L^\infty (\T,\mu)$, then
\begin{equation*}\textup{O} H=H\cdot \textup{O}1=HG\,,\end{equation*} say. Denote by $\|\cdot\|_{op}$ the operator norm, then
\begin{equation}\label{2inequality:op}
	\|HG\|_{L^2(\T,\mu)} = \|\textup{O} H\|_{L^2(\T,\mu)} \leq \|\textup{O}\|_{op}\|H\|_{L^2(\T\mu)} \,,
\end{equation}
which implies that $G \in L^\infty(\T,\mu)$ (otherwise one would be able to find a measurable set $B$ with positive measure on which $G$ is strictly larger than $\|\textup{O}\|_{op}$, and the indicator function $1_B$ would lead to a contradiction with \eqref{2inequality:op}).

 The ergodicity of $\mu$ with respect to $\textup{R}_\theta$ implies that only  constant functions in $L^\infty(\T,\mu)$ are $\textup{R}_\theta$-invariant $\mu$-a.e.. Hence, if $\textup{O}$ commutes with $\textup{T}_{\theta ,\mu }$ as well, then we can conclude that $\textup{O}$ is multiplication by a constant $c\in\C$.
By Schur's Lemma, the operators $\textup{T}_{\theta ,\mu },\textup{M}_\mu\in \mathcal B (L^2(\T,\mu))$ define an irreducible representation $\pi_{\theta ,\mu }$ of $\h$.
\end{proof}

Suppose that $\theta \in \T$ is irrational, and that $\mu$ and $\nu$ are two ergodic $\textup{R}_\theta $-quasi-invariant measures on $\T$.
Let $\pi_{\theta ,\mu }$ and $ \pi_{\theta ,\nu }$
be the corresponding irreducible unitary representations constructed above.

\begin{lemma}  \label{lemma:lemma2}
The representations $\pi_{\theta ,\mu }$ and $\pi_{\theta ,\nu }$ are unitarily equivalent if and only if $\mu$ and $\nu$ are equivalent.
\end{lemma}
\begin{proof}
Assume $\pi_{\theta ,\mu }$ and $\pi_{\theta ,\nu }$ are unitarily equivalent. Then there exists a unitary operator  $\textup{U}\colon L^2(\T,\mu) \longrightarrow L^2(\T,\nu)$ such that
\begin{equation}\label{2unitaryequi}
\textup{U} \pi_{\theta ,\mu }(\gamma)= \pi_{\theta ,\nu }(\gamma)\textup{U}
\end{equation}
for every $\gamma \in \h$.

Denote multiplication by a function $H\in C(\mathbb{T},\mathbb{C})$ by $\textup{O}_H$. The set of trigonometric polynomials, which is spanned by $\{\textup{M}_\mu ^n 1\,:\,n\in\Z\}$, is dense in $C(\T,\C)$. This implies
that \eqref{2unitaryequi} holds for all $H\in C(\T,\C)$, i.e., that $\textup{U} \textup{O}_H = \textup{O}_H  \textup{U}$ for any $H \in C(\T,\C)$.

Since $\textup{U}$ is an isometry we get that
\begin{align}\label{2riesz}
\int |H|^2 1^2 d\mu &=
\langle \textup{O}_H 1 , \textup{O}_H 1 \rangle_\mu
\\ &=
\langle  \textup{O}_H \textup{U}(1) ,  \textup{O}_H \textup{U}(1) \rangle_\nu
\\& = \int |H|^2 |\textup{U}(1)|^2 d\nu  \,,
\end{align}
where $\langle \cdot,\cdot\rangle_\sigma$ is the standard inner product on the Hilbert space $L^2 (\T,\sigma)$.
Using the same argument for $\textup{U}^{-1}$ we get, for every $H\in C(\mathbb{T},\mathbb{C})$,
\begin{equation}\label{2riesz2}
\int |H|^2 1^2 d\nu = \int |H|^2 |\textup{U}^{-1}(1)|^2 d\mu\,.
\end{equation}
Define, for every positive finite measure $\sigma$ on $\T$, a
linear
functional
\begin{equation*}
I_\sigma\colon C(\T,\C)\longrightarrow \C \quad \text{by}\quad I_\sigma(H)=\int H \,d\sigma\,.
\end{equation*}
 Since $I_\mu(H)=I_{|\textup{U}(1)|^2\nu}(H)$ and $I_\nu(H)=I_{|\textup{U}^{-1}(1)|^2\mu}(H)$ for all positive continuous functions $H$ by \eqref{2riesz} -- \eqref{2riesz2},  we conclude from the Riesz representation theorem that $\mu$ and $\nu$ are equivalent.

Conversely, if $\mu $ and $\nu $ are equivalent, then the
linear
operator
	\begin{displaymath}
	\textup{U}\colon L^2(\mathbb{T},\mu )\longrightarrow L^2(\mathbb{T},\nu )
	\quad
	\text{given by}
	\quad
\textup{U}F=\sqrt\frac{d\mu }{d\nu }F
	\end{displaymath}
for every $F\in L^2(\mathbb{T},\mu )$, is unitary and satisfies that $\textup{U}\pi _{\theta ,\mu } (\gamma )=\pi _{\theta ,\nu } (\gamma )\textup{U}$ for every $\gamma \in \h$.
\end{proof}

In this way one obtains uncountably many inequivalent irreducible unitary representation of $\h$ for a given irrational rotation number $\theta \in \mathbb{T}$.

In fact,  every irreducible unitary representation
$\pi$ of $\h$ with $\pi (z)=e^{2\pi i \theta}$, $\theta$ irrational, is unitarily equivalent to $\pi _{\theta ,\mu }$ for some probability measure $\mu $ on $\mathbb{T}$ which is quasi-invariant and ergodic with respect to an irrational circle rotation.
For convenience of the reader we sketch a proof of this fact based on elementary spectral theory of unitary operators.

Let $\pi$ be an irreducible unitary representation of $\h$ with representation space $\mathcal H_\pi$ such that $\pi(x)\pi(y)=e^{2 \pi i \theta}\pi(y)\pi(x)$. Let $v\in \mathcal H_\pi$ be a normalised cyclic vector, put $\textup{U}=\pi(y)$ and denote by $\mathcal H _v$ the closure of the subspace generated by $\{\textup U ^n v \,\colon \, n\in\Z \}$. The GNS-construction tells us that $a_n=\langle \textup U ^n v,v\rangle_{\mathcal H _\pi}$, $n\in \Z$, forms a positive-definite sequence. Due to Herglotz' (or, more generally, Bochner's) representation theorem there exists a probability measure $\mu_v$ on $\widehat \Z \simeq \T$ whose Fourier-Stieltjes transform $\widehat{\mu _v}$ fulfils
\begin{equation*}
\widehat{\mu _v}(n)=  \int_{\T} e^{- 2 \pi i n t} d \mu_v (t) = a_n \quad \text{for every $n \in \Z$.}
\end{equation*}
One easily verifies that there exists an isometric isomorphism
$\phi \colon \mathcal H _v \longrightarrow L^2 (\T,\mu_v)$ which intertwines the restriction $\textup U _v$ of $\textup {U}$ to $\mathcal H_v$ with the modulation operator $\textup M_v$ on $L^2 (\T,\mu_v)$ consisting of multiplication by $e^{ 2 \pi i  t}$. In other words, the unitary operators $\textup U _v$ and $\textup{M}_v$ are unitarily equivalent.

Put $\textup S =\pi(x)$ and consider the cyclic normalised vector $w=\textup S v$ of the representation $\pi$. By replacing $v$ by $w$ in the construction  above one can define the corresponding objects $\mathcal H _w, \textup U _w, \mu_w, L^2 (\T,\mu_w), \textup M_w$.

\begin{lemma}
The measures $\mu_v$ and $\mu_w$ are equivalent.
\end{lemma}
\begin{proof}
First note that $\textup U_v$ and $\textup U_w$ are unitarily equivalent. From this fact and the discussion preceding the lemma one concludes that
$\textup M _v$ and $\textup M _w$ are unitarily equivalent as well, i.e., that there exists a unitary operator $\textup O\colon L^2 (\T,\mu_v) \longrightarrow L^2 (\T,\mu_w)$ such that
$\textup O \textup M _v = \textup M _w \textup O$.
By arguing as in the first part of the proof of Lemma \ref{lemma:lemma2} one gets that $\textup O$ is a multiplication operator. Put $G=\textup O 1_{ L^2 (\T,\mu_v)}$.
Since $\textup O$ is an isometry one gets for all $\mu_v$-measurable sets $B$
\begin{equation*}
\mu_v(B)= \int_{\T} |1_B|^2 d \mu_v = \int_{\T} |G|^2|1_B|^2 d \mu_w \,.
\end{equation*}
By repeating these arguments with $v$ and $w$ interchanged
one concludes that $\mu_v$ and $\mu_w$ are equivalent. 
\end{proof}

\begin{lemma}
The measure $\mu_v$ is $\textup R _\theta $-quasi-invariant.
\end{lemma}
\begin{proof}
 Note that $\widehat {\mu_w}(n)=\langle \textup{S} ^{-1}\textup U ^n \textup S v,v\rangle_{\mathcal H _\pi} = e^{-2 \pi i \theta n} \widehat {\mu_v}(n)$ for every $n\in \mathbb{Z}$. As one can easily verify, for every probability measure $\mu$ on $\T$,
multiplying $\widehat \mu$ with a character $e^{-2 \pi i \theta n}$ is the same as the Fourier-Stieltjes transform of $\mu \circ \textup R _\theta$. Hence, we obtain that $\mu_w=\mu_v \circ \textup R _\theta$. As $\mu_v$ and $\mu_w$ are equivalent, $\mu_v$ is a $\textup R _\theta $-quasi-invariant probability measure on $\T$.
	\end{proof}

\begin{proof}[Completion of the proof of Theorem \ref{t:representations}]
The preceding discussion allows us to define an irreducible representation $\pi_v$ of $\h$ acting on
$ L^2 (\T,\mu_v)$ which is unitarily equivalent to $\pi$. The evaluation of $\pi_v$ at $y$ is given by $\textup M_v$ and  $\textup T_v= \pi_v(x)$ acts as composition of a translation operator by an angle $\theta$ and multiplication by some function $D_v\in L^\infty (\T,\mu_v)$. Due to the fact that  $\textup T_v$ has to be a unitary (and hence, an isometric) operator on $ L^2 (\T,\mu_v)$  the form of $D_v$ is fully determinted (cf. the definition of $\textup T _{\theta,\mu}$ in \eqref{2eq:T-theta-mu} for a $\textup R _\theta$-quasi-invariant measure $\mu$).  Since $\pi_v$ is irreducible only those multiplication  operators in $\mathcal B ( L^2 (\T,\mu_v))$  which act via multiplication by a constant function $c\in \C$ will commute with multiples of   the modified translation operator $\textup T_v$. This implies the ergodicity of $\mu_v$ and completes the proof of the theorem.
\end{proof}

\subsection{Wiener's Lemma for the discrete Heisenberg group}

Theorem \ref{2t:neumark} states that in order to decide on invertibility of $f \in \ell^1 (\h,\C)$,
one has to check invertibility of $\pi(f)$ for every irreducible representations $\pi$ of $\h$,
and hence for every $\pi_{\theta ,\mu }$ as above, where $\mu $ is a probability measure on $\mathbb{T}$ on $\mathbb{T}$
 which is quasi-invariant and ergodic with respect to a circle rotation $R_\theta $.

The problem of deciding on invertibility of $f\in \lh$ via Theorem \ref{2t:neumark} becomes much more straightforward if one is able to restrict oneself to unitary representations arising from \emph{rotation invariant} probability measures.
This is exactly our main result.

Before formulating this result we write down the relevant representations explicitly.
Take $\theta\in\mathbb T$, and consider the corresponding rotation $\textup{R}_\theta\colon \T\longrightarrow\T$ given by \eqref{2eq:Rtheta}. If $\theta $ is irrational, the Lebesgue measure $\lambda =\lambda _\mathbb{T}$ on $\mathbb{T}$ is the unique $\textup{R}_\theta $-invariant probability measure, and the representation $\pi _{\theta ,\lambda }$ on $L^2(\mathbb{T},\lambda )$ defined in \eqref{2eq:pimu} is irreducible. One can modify this representation by setting, for every $s,t\in \mathbb{T}$,
	\begin{equation}
	\label{2eq:pi_theta_s0}
\pi _{\theta }^{(s,t)}(x) = e^{2\pi is}\pi _{\theta ,\lambda }(x),\quad \pi _{\theta  }^{(s,t)}(y) =e^{2\pi it}\pi _{\theta ,\lambda }(y),
\quad \pi _{\theta }^{(s,t)}(z)=e^{2\pi i\theta } \,.
	\end{equation}
Then $\pi _{\theta  }^{(s,t)}$ is obviously again an irreducible unitary representation of $\h$ on $\mathcal H_{\pi^{(s,t)}_\theta}= L^2(\mathbb{T},\lambda )$.

If $\theta $ is rational we write it as $\theta =p/q$ where $p,q$ are integers with the properties $0\le p <q$ and $\gcd(p,q)=1$ and consider the unitary representation $\pi_\theta ^{(s,t)}$ of $\h$ on $\mathcal H_{\pi^{(s,t)}_\theta}=\mathbb{C}^q$ given by
	\begin{align}
	\label{2eq:pi_theta_s2}
\pi _\theta ^{(s,t)}(x) &= e^{2\pi i s}
	\begin{pmatrix}
	0	& I_{q-1} \\
	1	& 0
	\end{pmatrix}\,,
	\\ \label{2eq:pi_theta_s22}
	 \pi _\theta ^{(s,t)}(y) &= e^{2\pi i t}
	\left(\begin{smallmatrix}
	1	& 0	& \dots	 & 0& 0      \\
	0	& e^{2\pi i\theta }&\dots 	&0  & 0 \\
	\vdots	& \vdots 	& \ddots &\vdots & \vdots \\
	0 &0& \ldots &e^{2\pi i(q-2)\theta }&0\\
	0 	& 0& \dots & 0	 & e^{2\pi i(q-1)\theta }
	\end{smallmatrix}\right) \quad \text{and} \quad \pi^{(s,t)}_\theta(z)= e^{2 \pi i \theta}I_q\,,
	\end{align}
with $s,t\in \mathbb{T}$, where $I_{q-1}$ is the $(q-1)\times (q-1)$ identity matrix.
Every $\textup{R}_\theta $-invariant and ergodic probability measure $\mu $ on $\mathbb{T}$ is uniformly distributed on the set $\{t,1/q+t,\dots ,t+(q-1)/q\}\subset \mathbb{T}$ for some $t\in \mathbb{T}$; if we denote this measure by $\mu _t$ then $\mu _t=\mu _{t+k/q}$ for every $k=0,\dots ,q-1$.

With this notation at hand we can state our main result, the proof of which will be given in  Section \ref{2section4}.

	\begin{theorem}
	\label{2t:main}
An element $f\in\ell^1(\h,\C)$ is invertible if and only if the linear operators $\pi_{\theta }^{(s,t)}(f)$ are invertible on the corresponding Hilbert spaces $\mathcal{H}_{\pi_{\theta }^{(s,t)}}$ for every $\theta ,s,t\in\T$.
	\end{theorem}

The main advantage of Theorem \ref{2t:main} over Theorem \ref{2t:neumark} is that it is not necessary to check invertibility of $\pi (f)$ for \textit{every} irreducible representation of $\h$, but that one can restrict oneself for this purpose to the `nice' part of the dual of the non-Type I group $\h$. As we shall see later, one can make a further reduction if $\theta $ is irrational: in this case one only has to check invertibility of $\pi _{\theta }(f)= \pi _{\theta  }^{(1,1)}(f)$ on $L^2(\mathbb{T},\lambda )$.

\section{Wiener's Lemma for the discrete Heisenberg group: A proof and a first application}\label{2section4}

In this section we will present one possible proof of Wiener's Lemma for $\ell^1(\h, \C)$ and $C^\ast(\h)$.
This proof has two main ingredients, namely:
\begin{itemize}
\item Allan's local principle, which reduces the problem of verifying invertibility in $\ell^1(\h, \C)$ and $C^\ast(\h)$ to
the study of invertibility in rotation algebras.
\item The fact that irrational rotation algebras are simple will eliminate the `non-Type I problem' for questions about invertibility in $\ell^1(\h, \C)$ and $C^\ast(\h)$ discussed in the previous section.
\end{itemize}
These results will be generalised in Section \ref{2section5}  to group algebras of nilpotent groups.

\subsection{Local principles}
Let $\mathcal{A}$ be a unital Banach algebra and $a \in \mathcal A$.
Local principles are based on the following idea: one checks invertibility of projections of $a$ onto certain quotient algebras of $\mathcal A$ in order to conclude from this information whether $a$ is invertible or not.
Therefore, the main task is to find a sufficient family $\mathfrak J$ of ideals of $\mathcal A$ such that one can deduce the invertibility of $a$ from the invertibility of the projections of $a$ on $\mathcal A/\mathtt I$ for all $\mathtt I \in \mathfrak J$.

Allan's local principle provides us with such a sufficient family of ideals in case the centre of $\mathcal A$ is large enough.
We have used Allan's local principle already in \cite{2GSV} to study invertibility in $\ell^1(\h,\C)$. However, in that paper we were not able to prove Theorem \ref{2t:main} with this approach.

Suppose $\mathcal{A}$ is a unital Banach algebra with non-trivial centre
\begin{equation*}
	\mathcal C (\mathcal{A}) \coloneqq \bigl\{c \in \mathcal A \,:\, cb=bc \,\,\,\textup{for all }b \in \mathcal A\bigr\}\,.
\end{equation*}
The commutative Banach subalgebra $\mathcal C (\mathcal{A})$ is closed and contains the identity $1_{\mathcal A}$.
For every $m \in \textup{Max} (\mathcal C(\mathcal{A}))$ (the space of maximal ideals of $\mathcal{C}(\mathcal{A})$) we denote by $\mathtt{J}_m$ the smallest closed two-sided ideal of $\mathcal{A}$ which contains $m$ and denote by $\Phi_m\colon a\mapsto \Phi _m(a)$ the canonical projection of an element $a\in \mathcal{A}$ to the quotient algebra $\mathcal{A}/\mathtt{J}_m$. The algebra $\mathcal{A}/\mathtt{J}_m$, furnished with the quotient norm
\begin{equation}
	\|\Phi_m (a)\| \coloneqq \inf_{b\in \mathtt J_m} \|a + b\|_{\mathcal A} \,
\end{equation}
becomes then a unital Banach algebra.

\begin{theorem}[\cite{2Allan} Allan's local principle]
An element $a \in \mathcal{A}$ is invertible in $\mathcal{A}$ if and only if $\Phi_m(a)$ is invertible in $\mathcal{A}/\mathtt{J}_m$ for every $m\in \textup{Max} (\mathcal C(\mathcal{A}))$.
\end{theorem}

We would like to mention already here that in Section \ref{2section7} Allan's local principle will appear a second time and will link invertibility
of $f\in\lh$ to the invertibility of the evaluations of Stone-von Neumann representations at $f$.

Let us now prove our main theorem.

\subsection{Proof of Wiener's Lemma}\label{2subs:proofofwiener}
We apply the general observations made in the previous subsection to explore invertibility in $\lh$ and $C^\ast(\h)$.
Since $\lh$ is inverse-closed in $C^\ast(\h)$ we can focus on the study of invertibility in $C^\ast(\h)$.

Due to Allan's local principle
we have to check invertibility only in $\mathcal Q_\theta =C^\ast(\h)/\mathtt J_\theta $ for all $\theta \in \mathbb{T}$, where
${\mathtt J} _\theta =\overline{(z-e^{2\pi i\theta })C^\ast(\h)}$.
Indeed, $\mathcal C (C^\ast(\h)) \simeq C(\T ,\C)$, and the maximal ideals of  $ C(\T ,\C)$ are given by the sets
\begin{equation*}
	m_\theta \coloneqq \{F\in C(\T,\C) \,:\, F(\theta)=0 \}\,
\end{equation*}
and $\mathtt{J}_{m_\theta} = \mathtt J _\theta$.

Since ${\mathtt J} _\theta =\overline{(z-e^{2\pi i\theta })C^\ast(\h)} $ is a two-sided closed ideal we know that the quotient $\mathcal Q_\theta$ is a $C^\ast$-algebra and hence symmetric for each $\theta \in \T$.

By Schur's Lemma, if $\pi$ is an irreducible unitary representation of $\h$, then $\pi(z)= e^{2\pi i\theta }1_{\mathcal B (\mathcal H _\pi)}$ for some $\theta \in \mathbb{T}$. Hence, $\mathtt J_\theta$ is
a subset of $\ker (\pi)$ for every irreducible unitary representation $\pi$ of $\h$ with $\pi(z)= e^{2\pi i\theta }1_{\mathcal B (\mathcal H _\pi)}$.

If $\theta$ is rational the irreducible unitary representations of $\h$ vanishing on $\mathtt J _\theta $ are given by \eqref{2eq:pi_theta_s2} -- \eqref{2eq:pi_theta_s22} and were determined in \cite{2BEEK}. Due to the fact that $\mathcal Q_\theta$ is symmetric we can apply Theorem \ref{2t:neumark} in order to study invertibility in $\mathcal Q_\theta$ via the representations \eqref{2eq:pi_theta_s2} -- \eqref{2eq:pi_theta_s22}.

Now suppose $\theta $ is irrational. In order to study the representation theory of the $C^\ast$-algebra $\mathcal Q_\theta $ we have to understand the link to
one of the most studied non-commutative $C^\ast$-algebras -- the irrational rotation algebras.

We call a $C^\ast$-algebra an \textit{irrational rotation algebra} if it is generated by two unitaries $\textup{U},\textup{V}$ which fulfil the commutation relation
\begin{equation}\label{2d:rotation}
\textup{U}\textup{V}=e^{2\pi i\theta }\textup{V}\textup{U} \,,
\end{equation}
for some irrational $\theta \in \mathbb{T}$. We already saw examples of irrational rotation algebras above, namely, the $C^\ast$-subalgebras of $\mathcal B (L^2(\T,\mu))$ which are generated by $\textup{M}_\mu$
and $\textup{T}_{\theta ,\mu }$, where $\mu$ is a $\textup{R}_\theta$-quasi-invariant and ergodic measure. The reason why we call all $C^\ast$-algebras which fulfil \eqref{2d:rotation} irrational rotation algebras with parameter $\theta $ is the following striking result which can be found in \cite[Theorem VI.1.4]{2D} (and which was already proved in the 1960s, cf. \cite{2Rieffel} for a list of references).
\begin{theorem}\label{2t:simple}
If $\theta \in \mathbb{T}$ is irrational, then all $C^\ast$-algebras which are generated by two unitaries $\textup{U},\textup{V}$ satisfying \eqref{2d:rotation}, are $^\ast$-isomorphic.
\end{theorem}
We will denote the irrational rotation algebra with parameter $\theta $ by $\mathcal R _\theta $ and will not distinguish between the different realisations of $\mathcal R _\theta $ because of the universal property described in Theorem \ref{2t:simple}. Let us further note that the proof of Theorem
\ref{2t:simple} is deduced from the simplicity of the universal irrational rotation algebra.

The $C^\ast$-algebra $\mathcal Q_\theta $ is clearly a rotation algebra with parameter $\theta $. The simplicity of $\mathcal R _\theta $ implies that $\mathtt J _\theta $ is a maximal two-sided ideal of $C^\ast(\h)$. Hence, there exists an irreducible representation $\pi$ of $\h$ such that $\ker (\pi )=\mathtt J _\theta $, since every two-sided maximal ideal is primitive (cf. \cite[Theorem 4.1.9]{2Palmer1}).  Moreover, all the irreducible representations $\pi$ vanishing on $\mathtt J _\theta $ have the same kernel: otherwise we would get a violation of the maximality of $\mathtt J _\theta $. These representations are not all in the same unitary equivalence class (as we saw in Section \ref{2section3}), which is an indication of the fact that $\h$ is not of Type I.

	\begin{proof}[Proof of Theorem \ref{2t:main}]
	First of all recall that $\lh$ is inverse-closed in $C^\ast(\h)$.
By applying Allan's local principle for $C^\ast(\h)$ the problem of verifying invertibility in $\lh$ and $C^\ast(\h)$ reduces to the study of invertibility in the $C^\ast$-algebras $\mathcal Q_\theta$, with $\theta \in\T$.

The rational case is trivial and was already treated at the beginning of the discussion.

If $\theta $ is irrational, any irreducible representation $\pi$ of $\h$  which vanishes on $\mathtt J_\theta $ can be used  to check invertibility in $\mathcal Q_\theta $. Indeed, since
for an arbitrary unital $C^\ast$-algebra $\mathcal A$ and an irreducible representation $\pi$ of $\mathcal A$, the $C^\ast$-algebras $\pi (\mathcal A)$ and $\mathcal A /\ker ( \pi )$ are isomorphic one gets
\begin{equation*}
\pi(C^\ast(\h)) \simeq C^\ast(\h) / \ker(\pi) =  C^\ast(\h) /  \mathtt J _\theta           = \mathcal Q _\theta
\end{equation*}
due to Theorem \ref{2t:simple}.
In particular we may use the representations $\pi _{\theta}^{(1,1)}$  as in \eqref{2eq:pi_theta_s0}.
	\end{proof}

\begin{remark}
We should note here that for  all realisations of the irrational rotation algebra the spectrum of $a \in \mathcal R_\theta $ is the same as a set. But this does \textit{not} imply that an eigenvalue (or an element of the continuous spectrum) of $a$ in one realisation is an eigenvalue (or an element of the continuous spectrum) of $a$ in all the other realisations.
\end{remark}

\subsection{Finite-dimensional approximation}
The following proposition follows from Theorem \ref{2t:main} and might be useful for checking invertibility of $f\in\Z[\h]$ via numerical simulations.
\begin{proposition}\label{2prop}
Let $f\in\Z[\h]$. Then $\alpha_f$ is expansive if and only if there exists a constant $c>0$ such that $\pi (f)$ is invertible and
$\|\pi(f)^{-1}\| \leq c$ for every finite-dimensional irreducible representation $\pi$ of $\h$.
\end{proposition}
For the proof of the Proposition we work with the representations $\pi _{\theta}^{(1,1)}$ in \eqref{2eq:pi_theta_s0}. For irrational $\theta $,
\begin{equation}\label{2d:reponS}
	(\pi_{\theta}^{(1,1)}(x) H)(t)=H(t+\theta ), \quad (\pi_{\theta}^{(1,1)}(y) H)(t)=e^{2\pi i t}H(t)\,,
\end{equation}
for every $H \in L^2(\mathbb{T},\lambda_\mathbb{T})$ and $t\in\T$. For rational $\theta $ of the form $\theta =p/q$ with $(p,q)=1$  we replace the Lebesgue measure $\lambda =\lambda _\mathbb{T}$ in \eqref{2d:reponS} by the uniform
probability measure $\nu_q$ concentrated on the cyclic group $\{1/q,\dots ,(q-1)/q,1\}\subset \mathbb{T}$.

\begin{proof}
One direction is obvious. For the converse, assume that $\alpha_f$ is non-expansive, but that there exists a constant $c>0$ such that $\pi(f)$
is invertible and $\|\pi(f)^{-1}\| \leq c$ for every finite-dimensional irreducible representation $\pi$ of $\h$.

Since $\alpha_f$ is non-expansive, there exists an irrational $\theta$ (by our assumption) such that the operator $\pi_{\theta}^{(1,1)}(f)$ has no bounded inverse due to Theorem \ref{2t:main} and its proof. Therefore, $\pi_{\theta}^{(1,1)}(f)$ is either not bounded from below or its range is not dense in the representation space or both.

We consider first the case where $\pi_{\theta}^{(1,1)}(f)$ is not bounded from below. Then there exists, for every $\varepsilon >0$, an element $H_\varepsilon \in L^2(\mathbb{T},\lambda _\mathbb{T})$ with $\|H_\varepsilon \|_2 =1$ and $\|\pi_{\theta}^{(1,1)}(f)H_\varepsilon \|_2 < \varepsilon$. By approximating the $H_\varepsilon$ by continuous functions we may obviously assume that each $H_\varepsilon$ is continuous.

Let $q$ be a rational prime, and let $p$ satisfy $|\theta -p/q|<1/q$.
Then
\begin{equation*}
	\int |H_\varepsilon |^2 d\nu_q \quad \text{and} \quad  \int |\pi_{\theta}^{(1,1)}(f) H_\varepsilon |^2 d\nu_q
\end{equation*}
are Riemann approximations to the corresponding integrals with respect to $\lambda$. Hence,
\begin{equation*}
	\lim_{q \to \infty} \int |H_\varepsilon |^2 d\nu_q =1 \quad \text{and} \quad
	\lim_{q \to \infty} \int |\pi_{\theta}^{(1,1)}(f) H_\varepsilon |^2 d\nu_q \leq \varepsilon^2\,.
\end{equation*}
Furthermore, as $q\to\infty$, $\pi_{p/q}^{(1,1)}(f) H_\varepsilon$ converges uniformly to $\pi_{\theta}^{(1,1)}(f) H_\varepsilon$.
From this we deduce that
\begin{equation*}
\limsup_{q \to \infty} \int |\pi _{p/q}^{(1,1)}(f) H_\varepsilon |^2 d\nu_q \leq \varepsilon^2\,.
\end{equation*}
This clearly violates the hypothesis that $\pi _{p/q}^{(1,1)}(f)$, $q$ prime, have uniformly bounded inverses.

Finally, assume that $\pi _{\theta}^{(1,1)}(f)$ has no dense image in $L^2(\T,\lambda)$. In that case the adjoint operator  $(\pi_{\theta}^{(1,1)}(f))^\ast=\pi_{\theta}^{(1,1)}(f^\ast)$ is not injective \footnote{\label{2footnote:adjoint} For an operator $\textup{A}$ acting on a Hilbert space $\mathcal H$ one has
$(\ker \textup{A})^\perp = \overline{ \text{im } \textup{A}^\ast }$.}.
Furthermore, by our assumptions, $\|\pi(f^\ast)^{-1}\| \leq c$ for every finite-dimensional irreducible representation $\pi$ of $\h$.
The same arguments as in the first part of the proof lead to a contradiction.
\end{proof}

\section{Invertibility in group algebras of discrete nilpotent groups} \label{2section5}
In this section we aim to find more evident conditions for invertibility in group algebras for discrete countable  nilpotent groups
than the one given in Theorem \ref{2t:neumark}.
 The main objects of our investigations are the primitive ideal space and the class of irreducible monomial representations of the group.

\subsection{Wiener's Lemma for nilpotent groups}
Let $\Gamma$ be a countable discrete nilpotent group.
As we have seen earlier, $\ell^1(\Gamma,\C)$ is inverse-closed in $C^\ast(\Gamma)$. Hence we concentrate on the group $C^\ast$-algebra $C^\ast(\Gamma)$.

In order to establish a Wiener Lemma in this more general setting we are first going to reinterpret Wiener's Lemma for the discrete Heisenberg group.
From the discussion in Subsection \ref{2subs:proofofwiener} one can easily see that the irreducible unitary representations $\pi_\theta^{(s,t)}$, $\theta,s,t\in\T$, those representations which correspond to ergodic $\textup{R}_\theta$-invariant measures on $\T$, generate the primitive ideal space $\textup{Prim}(C^\ast(\h))$. Moreover, since $\pi(C^\ast(\h)) \simeq C^\ast(\h) / \ker(\pi)$ for every $\pi\in\widehat \h$ the study of invertibility is directly linked to invertibility of projections onto the primitive ideals. We may interpret this as a localisation principle.

Before formulating a Wiener Lemma for an arbitrary discrete nilpotent group let us fix some notation. Let $\mathcal A$ be a unital $C^\ast$-algebra. For every  two-sided closed ideal $\mathtt J$ of $\mathcal A$, denote by $\Phi_{\mathtt J}$  the canonical projection from $\mathcal A$ onto the $C^\ast$-algebra $\mathcal A/ \mathtt J$.

\begin{theorem}[Wiener's Lemma for nilpotent groups]\label{2t:Wienerviaprimitive}
If $\Gamma$ is a discrete nilpotent group, then $a \in C^\ast(\Gamma)$ is invertible if and only if $\Phi_\mathtt{I}(a)$ is invertible for every $\mathtt I \in \textup{Prim}(C^\ast(\Gamma))$.
\end{theorem}
This theorem links questions about invertibility in $\ell^1(\Gamma,\C)$ and $C^\ast(\Gamma)$ to their representation theory and, to be more specific, to the primitive ideal space $\textup{Prim}(C^\ast(\Gamma))$. At the same time this result provides us with a sufficient family of ideals in order to study invertibility and hence, Wiener's Lemma for nilpotent groups describes a localisation principle.
We will learn in the next subsection that for discrete nilpotent groups $\Gamma$ the class of irreducible representation which are induced by one-dimensional representations of subgroups of $\Gamma$ provide us with an effective tool to generate $\textup{Prim}(C^\ast(\Gamma))$. In other words, it is a feasible task to determine the primitive ideal space $\textup{Prim}(C^\ast(\Gamma))$.

Theorem \ref{2t:Wienerviaprimitive} can  be generalised to all unital $C^\ast$-algebras. Moreover,
we provide a sufficient condition for a family of ideals in order to check  invertibility via localisations.

Suppose $\mathfrak I$ is a collection of ideals of a $C^\ast$-algebra $\mathcal A$, such that
\begin{enumerate}
\item[(i)] every $\mathtt I\in\mathfrak I$ is closed and two-sided,
\item[(ii)] for any primitive ideal $\mathtt J\in \text{Prim}(\mathcal A)$ there exists $\mathtt I\in\mathfrak I$ such that
$\mathtt I\subseteq \mathtt J$.
\end{enumerate}

\begin{theorem}\label{2t:closedideals}
Let $\mathcal A$ be a unital $C^\ast$-algebra.
Suppose $\mathfrak I$ satisfies conditions \textup{(i)} and \textup{(ii)} above.
Then an element $a$ in $\mathcal A$ is invertible if and only if
for every $ \mathtt I\in\mathfrak I$ the projection of
$a$ on $\mathcal A/\mathtt I$ is invertible.
\end{theorem}
By setting $\mathfrak I=\text{Prim}(C^*(\Gamma))$,  Theorem \ref{2t:Wienerviaprimitive} just becomes  a particular case of Theorem \ref{2t:closedideals}.

\begin{proof}
If $a \in \mathcal A$ is not invertible, then by Theorem \ref{2t:neumark} there exists an irreducible unitary representation $\pi$ of $\mathcal A$ such that
$\pi(a)v=0$ for some non-zero vector $v\in \mathcal H_ \pi$.
Moreover, for every two-sided closed ideal $\mathtt I \subseteq \ker (\pi )$ of $\mathcal A$
the representation $\pi$ induces a well-defined irreducible representation $\pi_{\mathtt I}$ of the $C^\ast$-algebra $\mathcal A/ \mathtt I$, i.e.,
\begin{equation*}
\pi_{\mathtt I}(\Phi_{\mathtt I}(a))=\pi(a) \,.
\end{equation*}
Hence, for every two-sided closed ideal $\mathtt I \subseteq \ker (\pi)$ of $C^\ast(\h)$, the element $\Phi_{\mathtt I}(a)$ is not invertible in $\mathcal A/ \mathtt I$, since the operator $\pi_{\mathtt I}(\Phi_{\mathtt I}(a))$ has a non-trivial kernel in $\mathcal H_\pi$.

Let us assume now that $\Phi_{\mathtt I}(a)$ is not invertible in the $C^\ast$-algebra $\mathcal A/ \mathtt I$ for some $\mathtt I\in\mathfrak I$. Hence, there exists an irreducible representation $\rho$ of $\mathcal A/ \mathtt I$ such that \begin{equation*}\rho (\Phi_{\mathtt I}(a)) v =0\end{equation*} for some vector
$v\in\mathcal H _\rho$. The irreducible representation $\rho$  can be extended to an irreducible representation $\tilde \rho$ of $\mathcal A$ which vanishes on $\mathtt I$ and which is given by $\tilde \rho = \rho \circ \Phi_{\mathtt I}$. Therefore, $a$ is not invertible in $\mathcal A$.
\end{proof}
From the proof of Theorem \ref{2t:closedideals} we get the following corollary.
\begin{corollary}\label{2tc:closedideals}
If $\pi(a)$ is not invertible for an irreducible representation $\pi$, then for every two-sided closed ideal $\mathtt I \subseteq \ker (\pi)$ of $C^\ast(\h)$, the element $\Phi_{\mathtt I}(a)$ is non-invertible in $\mathcal A/ \mathtt I$.
\end{corollary}

\begin{example} Denote by $\mathtt J_\theta$, with $\theta\in\T$, the localisation ideals of $C^\ast(\h)$ as defined in Section \ref{2section4} and set
\begin{equation*}
\mathfrak J_{\h} \coloneqq \{\mathtt J _\theta \,\colon\,\theta\in \T\}\,.
\end{equation*}
Obviously, the restriction $r\colon\textup{Prim}(C^\ast(\h)) \longrightarrow \mathfrak J_{\h}$ given by $\mathtt I \mapsto r(\mathtt I)=\mathtt I \cap \mathcal C (C^\ast(\h))$ for every primitive ideal $\mathtt I$ of $C^\ast(\h)$, defines a surjective map. Hence,
Theorem \ref{2t:closedideals}  provides  a proof of Allan's local principle for $C^\ast(\h)$.
Moreover, Allan's local principle can be viewed as the most effective way to apply Theorem \ref{2t:closedideals} in order to check invertibility.
\end{example}

\subsection{Monomial representations}
\subsubsection*{The Heisenberg group}
Denote by $\textup{Ind}^{\h}_N(\sigma_{\theta,s})$ the representation of $\h$ induced from the
	normal subgroup
	\begin{equation*}
		N\coloneqq
		\left\{ \left ( \begin{matrix}
		1 & a & b \\
		0 & 1 & 0 \\
		0 & 0 & 1
		\end{matrix} \right) \,:\, a,b \in \Z\right\}
	\end{equation*}
	and the character $\sigma_{\theta,s}$ which is defined by
	\begin{equation*}
		 \sigma_{\theta,s}(z)=e^{2 \pi i \theta} \quad \text{and} \quad \sigma_{\theta,s}(x)=e^{-2 \pi i s}\,.
	\end{equation*}
	For the convenience of the reader we will write down  $\textup{Ind}^{\h}_N ( \sigma_{\theta ,s} )$ for every $\theta , s \in \T$	explicitly
	\begin{equation}\label{2e:inducedstonevonneumann}
	\left(\textup{Ind}^{\h}_N ( \sigma_{\theta ,s}) (x^k y^l z^m)F \right)(n) = e^{2 \pi i (m \theta - k(n\theta+s) )} F(n+l)
	\end{equation}
	for all $k,l,m,n\in \Z$ and $F\in \ell^2(\Z,\C)$.

	The representations $\textup{Ind}^{\h}_N(\sigma_{\theta,s})$ play a special role since they can be extended to the Stone-von Neumann representations of the \textit{real} Heisenberg group $\h_{\R}$ consisting of all unipotent upper triangular matrices in $\textup{SL}(3,\mathbb{R})$.
 The Stone-von Neumann representations of $\h_{\R}$ are obtained from Mackey's induction procedure from the real analogue of $N$, i.e.,
 \begin{equation*}
		N_{\R} \coloneqq
		\left\{ \left ( \begin{matrix}
		1 & a & b \\
		0 & 1 & 0 \\
		0 & 0 & 1
		\end{matrix} \right) \,:\, a,b \in \R\right\}
	\end{equation*}
	and its characters. The Stone-von Neumann representations are defined by modulation and translation operators on $L^2(\R,\C)$.

It is easy to see that for irrational $\theta$
the representation $\pi_{\theta}^{(1,1)}$ in \eqref{2eq:pi_theta_s0} is unitarily equivalent (via Fourier transformation) to the representation
$\textup{Ind}^{\h}_N(\sigma_{\theta,1})$. Moreover, every irreducible finite dimensional representation of a nilpotent group $\Gamma$ is induced from a one dimensional representation of a subgroup of $\Gamma$ (cf. \cite[Lemma 1]{2Brown}).

Therefore, the monomial representations contain all representations involved in
validating invertibility via Theorem \ref{2t:main}.

The natural question arises, whether one can always restrict oneself to the class of monomial representations of $\Gamma$ when analysing invertibility in the corresponding group algebras, in case $\Gamma$ is a countable discrete nilpotent group. We will show that the answer is positive.

\subsubsection*{The general case}
Let $\Gamma$ be a countable discrete nilpotent group. Define an equivalence relation on $\widehat \Gamma$ as follows:
\begin{equation*}
	\pi_1 \sim \pi_2 \iff \ker (\pi_1 ) = \ker (\pi _2 )\,,
\end{equation*}
where $\pi_1,\pi_2$ are irreducible unitary representations of $\Gamma$. This equivalence relation is the same as the notion of weak equivalence according to \cite{2Fell}.

The next theorem was established by Howe in \cite[Proposition 5]{2Howe}.
\begin{theorem}\label{2t:howe}
Suppose that $\Gamma$ is a countable discrete nilpotent group. Then every irreducible unitary representation is weakly equivalent  to an irreducible  monomial representation of $\Gamma$.
\end{theorem}
In other words the map from the subclass of irreducible monomial representations to the primitive ideal space is surjective and as a conclusion the monomial representations generate the primitive ideal space. It is therefore not surprising that the class of irreducible monomial representations contains all the information which is necessary in order to study invertibility in the group algebras. As we will show, combining  Theorem \ref{2t:howe} with Theorem \ref{2t:neumark} leads to another Wiener Lemma:

\begin{theorem}\label{monominv}
An element $a\in C^\ast(\Gamma)$ is non-invertible if and only if there exists an irreducible monomial representation $\pi$ such that $\pi(a)$ has no bounded inverse.
\end{theorem}
For convenience of the reader we explain the ideas once more.
\begin{proof}
If $a$ is not invertible, then there exists an irreducible unitary representation $\pi$ and a non-zero vector $v\in\mathcal H _\pi$ such that $\pi(a)v=0$.
This implies that $\Phi_{\ker (\pi )}(a)$ is not invertible in $C^\ast(\Gamma)/\ker (\pi )$. Moreover, there exists an irreducible monomial representation $\rho$ with
$\ker (\rho )=\ker( \pi )$ (cf. Theorem \ref{2t:howe}) and hence
\begin{equation*}
\pi( C^\ast(\Gamma) ) \simeq C^\ast(\Gamma) / \ker (\pi )=  C^\ast(\Gamma) / \ker (\rho) \simeq\rho(C^\ast(\Gamma) ) \,.
\end{equation*}
Therefore, $\Phi_{\ker (\rho )}(a)$ and $\rho(a)$ are not invertible.

On the other hand, if $\pi(a)$ is not invertible for an irreducible monomial representation $\pi$, then $\Phi_{\ker (\pi ) }(a)$ is not invertible in the $C^\ast$-algebra
$C^\ast(\Gamma) / \ker (\pi)$.  Hence there exists an irreducible representation $\rho$ of $C^\ast(\Gamma) / \ker (\pi)$ such that $\rho(\Phi_{\ker (\pi )}(a))$ has a non-trivial kernel. Moreover, $\rho$ can be extended to a representation $\tilde \rho$ of $C^\ast(\Gamma)$ vanishing on $\ker (\pi )$. Therefore, $a$ is not invertible.
\end{proof}

\subsection{Maximality of primitive ideals}
In the previous subsection we saw that we can restrict our attention to irreducible monomial representations  for questions about invertibility.
Unfortunately, this subclass of irreducible representations might still be quite big.
We will use another general result about the structure of $\textup{Prim}(C^\ast(\Gamma))$ to make the analysis of invertibility in $C^\ast(\Gamma)$
easier.
\begin{theorem}[\cite{2Poguntke}] \label{2t:prim=max}
Let $\Gamma$ be a discrete nilpotent group. Then
\begin{equation*}
\textup{Prim}(C^\ast(\Gamma))=\textup{Max}(C^\ast(\Gamma))\,,
\end{equation*}
i.e., every primitive ideal of $C^\ast(\Gamma)$ is maximal.
\end{theorem}
The simplification in the study of invertibility in $C^\ast(\h)$ was due to the simplicity of the
irrational rotation algebras $\mathcal R _\theta $,
which is equivalent to the maximality of the two-sided closed ideal $\mathtt J _\theta$.
We should note here that Theorem \ref{2t:simple} is usually proved by the construction of a unique trace on $\mathcal R _\theta $, which is rather complicated. Alternatively, let $\theta \in \T$ be irrational. Then it easily follows from Theorem \ref{2t:prim=max} and the fact that $\pi_{\theta}^{(s,t)}$ is an irreducible representation (cf. Lemma \ref{2lemma:irreducible}) with $\ker (\pi_{\theta}^{(s,t)}) = \mathtt J_\theta $ that
$\mathtt J_\theta $ is maximal. This is exactly the statement of Theorem \ref{2t:simple}.
In the next subsection we will see  applications of Theorem \ref{2t:prim=max}. It turns out that this representation theoretical result will eliminate the `non-Type I issues' exactly as the simplicity of irrational rotation algebras did for the group algebras of the discrete Heisenberg group.

\subsection{Examples}

The first example shows how to establish a Wiener Lemma for $\h$ from the general observation made in this section.
\begin{example}
Consider the monomial representations $\textup{Ind}_N^{\h}(\sigma_{\theta , s})$ of $\h$ as
 defined in \eqref{2e:inducedstonevonneumann} for irrational $\theta$ and arbitrary $s\in\T$.
 Obviously, one has for every $s\in \T$ that $\ker (\textup{Ind}_N^{\h}(\sigma_{\theta , s})) =\mathtt J_\theta$.

We will show that there is no bounded operator on $\ell^2(\Z,\C)$ which commutes with the operators $\textup{Ind}_N^{\h}(\sigma_{\theta , s})(x)$ and $\textup{Ind}_N^{\h}(\sigma_{\theta , s})(y)$ except
multiples of the identity operator. Let $\{\delta_k \,:\, k \in \Z\}$ be the standard basis of  $\ell^2(\Z,\C)$ and $\textup C = (c_{n,k})_{n,k\in \Z}$ an operator which commutes with $\textup{Ind}_N^{\h}(\sigma_{\theta , s})(x)$ and $\textup{Ind}_N^{\h}(\sigma_{\theta , s})(y)$.
From the equations
\begin{align*}
	e^{-2 \pi i s}\sum_{n \in \Z} c_{n,k} e^{-2 \pi i \theta n} \delta_n
	&=\textup  C \left(\textup{Ind}_N^{\h}(\sigma_{\theta , s})(x)  \delta_k \right)
	\\& =\textup{Ind}_N^{\h}(\sigma_{\theta , s})(x) (\textup C  \delta_k)
	\\&= e^{-2 \pi i s}e^{-2 \pi i \theta k}\sum_{n \in \Z} c_{n,k}  \delta_n
\end{align*}
and the fact that $\theta$ is irrational we can conclude that
$c_{n,k}= 0$ for all $n, k \in \Z$ with $n \not = k$.
On the other hand, for $k\in\Z$
\begin{align*}
 c_{k,k} \delta_{k+1}
 &=\textup{Ind}_N^{\h}(\sigma_{\theta , s})(y)(\textup C  \delta_k)
  \\&
 = \textup C \left(\textup{Ind}_N^{\h}(\sigma_{\theta , s}) (y)  \delta_k\right)
 \\& = c_{k+1,k+1} \delta_{k+1}  \,.
\end{align*}
Therefore, the only operators in the commutant of $\textup{Ind}_N^{\h}(\sigma_{\theta , s}) (\h)$ are scalar multiples of the identity, which is equivalent to the irreducibility of the representation $\textup{Ind}_N^{\h}(\sigma_{\theta , s})$ by Schur's Lemma.
Hence, the kernel of the irreducible monomial representation $\textup{Ind}_N^{\h}(\sigma_{\theta , s})$ is
a maximal two-sided ideal (cf. Theorem \ref{2t:prim=max})
 given by  $\mathtt J _\theta$.

For every irreducible representation $\pi$ of $\h$ with $\mathtt J _\theta \subseteq \ker (\pi)$ one has
$\ker (\pi ) =\mathtt J _\theta$ due to the maximality of $\mathtt J _\theta$ which we deduce from the irreducibility of
$\textup{Ind}_N^{\h}(\sigma_{\theta , s})$.

 Consider $\theta =\frac{n}{d}$ with relatively prime $n,d \in \N$. We note that analysing invertibility in $\mathcal Q _\theta $ reduces to the study of monomial representations as well. Set
\begin{equation*}
\h / Z(d) \coloneqq
\left \{ \left( \begin{matrix}
	1 & a & \bar{b} \\
	0 & 1 & c \\
	0 & 0 & 1
\end{matrix}\right), \quad a,c \in \Z \,\text{and}\, \bar{b} \in \Z /d\Z \right\}\,, \quad d\in \N\,,
\end{equation*}
and note the isomorphism $\mathcal Q _\theta \cong C^\ast(\h / Z(d))$.
The nilpotent group $\h / Z(d)$ is of Type I since $\h / Z(d)$ has normal abelian subgroups of finite index, e.g.,
\begin{equation*}
\left \{ \left( \begin{matrix}
	1 & a d & \bar{b} \\
	0 & 1 & c \\
	0 & 0 & 1
\end{matrix}\right), \quad a,c \in \Z \,\text{and}\, \bar{b} \in \Z /d\Z \right\} \,.
\end{equation*}
Hence, all irreducible representations are monomial by Theorem \ref{2typeI/mono}.

A Wiener Lemma can now be deduced from Theorem \ref{2t:prim=max}.

Note that in the general study of invertibility in this example we have not used Allan's local principle or any results from Section \ref{2section4} explicitly.
\end{example}

We give another example of a group where Theorem \ref{2t:prim=max} simplifies the analysis.
\begin{example}
Let us denote by $\mathbb G$ the multiplicative group given by
\begin{equation*}
\left \{
\left(
	\begin{matrix}
1 & a & c & f
	\\
0 & 1 & b & e
	\\
0 & 0 & 1 & d
	\\
0 & 0 & 0 & 1
	\end{matrix}
\right)
\,:\, a,b,c,d,e,f \in \Z \right \} \,.
\end{equation*}
One can easily verify that the centre of this group is given by
\begin{equation*}
\left \{
\left(
	\begin{matrix}
1 & 0 & 0 & f
	\\
0 & 1 & 0 & 0
	\\
0 & 0 & 1 & 0
	\\
0 & 0 & 0 & 1
	\end{matrix}
\right)
\,:\, f \in \Z \right \} \simeq \Z,
\end{equation*}
and hence the corresponding  central sub-algebra $\ell^1(\Z)$
is exactly the same as it was in the case of the discrete Heisenberg group.
It is therefore not surprising that the invertibility problem can be
addressed in a similar fashion.

Let us construct monomial representations, which will be sufficient to check
global invertibility (cf. Theorem \ref{monominv}).

First note that $\mathbb G$ can be identified with the semi-direct product $G_1\rtimes G_2$ of the normal abelian subgroup
\begin{equation*}
G_1\coloneqq \left \{
\left(
	\begin{matrix}
1 & 0 & c & f
	\\
0 & 1 & b & e
	\\
0 & 0 & 1 & 0
	\\
0 & 0 & 0 & 1
	\end{matrix}
\right)
\,:\, b,c,e,f \in \Z \right \} \simeq \Z^4
\end{equation*}
and the  abelian subgroup
\begin{equation*}
G_2 \coloneqq \left \{
\left(
	\begin{matrix}
1 & a & 0 & 0
	\\
0 & 1 & 0 & 0
	\\
0 & 0 & 1 & d
	\\
0 & 0 & 0 & 1
	\end{matrix}
\right)
\,:\, a,d \in \Z \right \} \simeq \Z^2\,.
\end{equation*}
In such a situation the construction of induced representations becomes very easy. We refer to \cite[Section 2.4]{2KT} for all the details.
Now let $\sigma_{\theta_b,\theta_c,\theta_e,\theta_f}$ be the one-dimensional representation of $G_1$ given by
\begin{equation*}
\sigma_{\theta_b,\theta_c,\theta_e,\theta_f} \left( \left(
	\begin{matrix}
1 & 0 & c & f
	\\
0 & 1 & b & e
	\\
0 & 0 & 1 & 0
	\\
0 & 0 & 0 & 1
	\end{matrix}
 \right)\right)= e^{2 \pi i \theta_b b}e^{2 \pi i \theta_c c}e^{2  \pi i \theta_e e }e^{2 \pi i \theta_f f }\,.
\end{equation*}
The inclusion map from $G_2$ to $\mathbb G$ will serve as a cross-section.
The induced representation
$
\textup{Ind}_{G_1}^{\mathbb G}(\sigma_{\theta_b,\theta_c,\theta_e,\theta_f})
$
(is unitarily equivalent to a representation which) acts on the Hilbert space $\ell^2(\Z^2,\C)$ and is given by
\begin{equation}\label{2eq:ind}\begin{aligned}
&\left(\textup{Ind}_{G_1}^{\mathbb G}(\sigma_{\theta_b,\theta_c,\theta_e,\theta_f})\left(
\left(\begin{matrix}
1 & 0 & c & f
	\\
0 & 1 & b & e
	\\
0 & 0 & 1 & 0
	\\
0 & 0 & 0 & 1
	\end{matrix}
	\right)
	\left(
	\begin{matrix}
1 & a & 0 & 0
	\\
0 & 1 & 0 & 0
	\\
0 & 0 & 1 & d
	\\
0 & 0 & 0 & 1
	\end{matrix}
\right)\right) F\right)(k,l)
\\&= e^{2  \pi i \theta_b b} e^{2 \pi i \theta_c (c-kb)}e^{2  \pi i \theta_e (e+lb) }e^{2 \pi i \theta_f (f+lc-ke-klb) } F(k-a,l-d)\,,
\end{aligned}\end{equation}
for every $a,b,c,d,e,f,k,l \in \Z$ and $F\in \ell^2(\Z^2,\C)$.

The localisation fibres are indexed by $\theta_f$. It is clear that for every irrational $\theta_f$  and arbitrary $\theta_b, \theta_c, \theta_e$,
\begin{equation*}\ker \left( \textup{Ind}_{G_1}^{\mathbb G}(\sigma_{\theta_b,\theta_c,\theta_e,\theta_f})\right)= \mathtt J_{\theta_f}\,.\end{equation*}

In the case of irrational $\theta_f$, the commutant
of $\textup{Ind}_{G_1}^{\mathbb G}(\sigma_{\theta_b,\theta_c,\theta_e,\theta_f})(\mathbb G)$ in $\mathcal B(\ell^2(\Z^2,\C))$ is trivial which is equivalent to irreducibility by Schur's Lemma. Hence, for irrational $\theta_f$ the two-sided closed ideal $\mathtt J_{\theta_f}$ is maximal by Theorem \ref{2t:prim=max} and one has to consider only a single representation, e.g., the one given in \eqref{2eq:ind} for fixed parameters
$\theta_b,\theta_c,\theta_e$, to study invertibility on these fibres.
\end{example}

\subsection{A kernel condition and finite-dimensional representations}

 The  proof of Proposition \ref{2prop} provides an approximation argument which allows restricting oneself to finite-dimensional representations for the purpose of proving invertibility. This result can be reinterpreted as a density argument. The finite-dimensional irreducible representations of $\h$ correspond to `rational' points in the dual of $\h$. We know that the intersection of all irreducible representations $\pi$ with $\pi (z)=e^{2\pi i \theta}$ coincides with $\mathtt J _\theta$. In the same way as one shows that no non-zero element in $C(\T)$ (which is isomorphic to $C^*(\Z)$) vanishes at all rational points, one can prove that
$$
\bigcap _{\text{rational} \,\theta \in \T } \mathtt J_\theta = \{0\}\,.
$$
We will show that this empty-intersection condition actually implies that for checking invertibility of $a\in C^\ast(\h)$ it is enough to check
the invertibility of the evaluations  $\pi(a)$ for all finite-dimensional irreducible representations of $C^\ast (\h)$ \textit{provided that the inverses $\pi(a)^{-1}$ of these elements are uniformly bounded in norm}.

Let $\mathcal A$  and $\mathcal B _t$, for all $t\in I$ for some index set $I$,  be  $C^\ast$-algebras.
Let us denote by $\mathfrak F$ a family  of $^\ast$-morphisms $\phi_t \colon \mathcal A \longrightarrow \mathcal B_t$, $t\in I$, which fulfils
\begin{itemize}
\item for all $t\in I$ one has $\mathtt J_t = \ker (\phi _t )$ is a two-sided closed ideal of $\mathcal A$, hence $\mathcal A _t= \mathcal A /  \mathtt J_t  $ is a $C^\ast$-algebra with quotient norm $\|\cdot \|_t$;
\item $\bigcap _{t \in I} \mathtt J _t = \{0\}$.
\end{itemize}
Furthermore, let us denote by $\mathcal A_I$ the set of elements  $a=(a_t)_{t\in I} \in \prod_{t\in I} \mathcal A_t$ with $
\|a\|_I \coloneqq \sup_{t\in I}\|a_t \|_t< \infty$; $\mathcal A_I$ together with the norm $\|\cdot\|_I$ forms a $C^\ast$-algebra. Let $\Phi \colon\mathcal A \longrightarrow \mathcal A_I$ be defined by $a\mapsto (\phi_t(a))_{t\in I}$.
Since
$$
\bigcap _{t\in I } \mathtt J_t = \{0\}\,,
$$
one has that $\Phi$ is a bijective map from $\mathcal A$ to $\Phi(\mathcal A)$.
The $C^\ast$-algebras $\Phi (\mathcal A)$ and $\mathcal A_I$ form a Wiener pair. Hence, $a\in\mathcal A$ is invertible if and only if
 $\phi_t (a)$ is invertible for all $t\in I$ and
$\| \phi_t (a) ^{-1} \|_t$ is uniformly bounded in $t$.

\begin{example}
We apply  these ideas to $C^\ast (\h)$ and set $$\mathfrak F _{\h}= \{ \pi \in \widehat{\h} \,\colon\, \text{$\pi$ a finite dimensional representation}\}$$
in order to  get an algebraic interpretation of Proposition \ref{2prop}.
\end{example}
\begin{example}
Definitions and results that are  used in this example can be found in \cite{2Howe2} by Howe.
Consider a finitely-generated nilpotent torsion-free group $\Gamma$. The set of kernels of finite-dimensional representations forms a dense subset of $\textup{Prim}(C^\ast(\Gamma))$ with respect to the hull-kernel topology. Since all $C^\ast$-algebras are semi-simple one gets that for every dense subset $\mathfrak J \subseteq \textup{Prim}(C^\ast(\Gamma))$ the following holds
$$
\bigcap_{\mathtt J \in \mathfrak J}\mathtt J = \bigcap_{\mathtt J \in \overline{\mathfrak J}}\mathtt J= \bigcap_{\mathtt J \in \textup{Prim}(C^\ast(\Gamma))}\mathtt J =\{0\}\,.
$$
Hence, for verifying invertibility in $C^\ast(\Gamma)$, the study of the evaluations of the finite-dimensional representations -- combined with a boundedness condition -- is sufficient.

Suppose that $\Gamma$ is also elementarily-exponentiable --  Howe  labels such  groups  to have a well-defined `Lie-algebra', say $\mathcal L$.
 Then the finite dimensional representations correspond to finite quasi-orbits  of a canonical action of $\Gamma$ on $\mathcal L$
and the representation theory of $\Gamma$ is closely related to the one of its Mal'cev completion.
\end{example}

A systematic treatment of group-$C^\ast$-algebras $C^\ast(\Gamma)$ whose finite-dimensional representations separate points of $C^\ast(\Gamma)$ can be found in Section 4 of \cite{2BL}.

\section{A connection to Time-Frequency-Analysis via localisations}\label{2section6}
In this section we formulate yet another Wiener Lemma for $\ell^1(\h,\C)$ which involves Stone-von Neumann representations.
These unitary representations of the discrete Heisenberg group are highly reducible and therefore, not the first choice for invertibility investigations (cf. Theorem \ref{2t:neumark}). However, by exploring a connection from localisations of $\ell^1(\h,\C)$ to twisted convolution algebras we establish a link to  Time-Frequency-Analysis. In this discipline of mathematics Stone-von Neumann representations are of great importance.
\subsection{Localisations and twisted convolution algebras}
In \cite{2GSV} we determined the explicit form of the localisation ideals $\mathtt J _m$ in order to formulate Allan's local principle for the group algebra $\ell^1(\h,\C)$ of the discrete Heisenberg group. Let us recall this result.

We write a typical element $f\in \ell^1(\h,\C)$ in the form:
\begin{equation*}
	f=\sum_{(k,l,m)\in \Z^3} f_{(k,l,m)}x^k y^l z^m \,,
\end{equation*}
with $f_{(k,l,m)}\in \C$ and $\sum_{(k,l,m)\in \Z^3} |f_{(k,l,m)}| < \infty$.
We identify the centre of $\lh$ with $\ell^1(\Z,\C)$ since the centre
 of the group $\h$ is generated by powers of $z$.
The maximal ideal space $\textup{Max}(\ell^1(\Z,\C))$ is canonically homeomorphic to $\widehat{\Z}\cong\T$.
The smallest closed two-sided ideal in $\ell^1(\h,\C)$ containing $m_\theta \in \textup{Max}(\ell^1(\Z,\C))$, $\theta \in \mathbb{T}$, is given by the subset
$\mathtt J _\theta \subset \ell^1(\h,\C)$ which consists of all elements $f \in \ell^1(\h,\C)$ such that
	\begin{equation*}
		f^\theta \coloneqq  \sum_{(k,l,m)\in\Z^3} f_{(k,l,m)} x^k y^l e^{2\pi im\theta } = 0_{\ell^1(\h,\C)}\,.
	\end{equation*}

The next definition plays an important role in the field of Time-Frequency-Analysis.
Fix $\theta \in \T$.
The \textit{twisted convolution} $ \natural_\theta $ on $\ell^1(\Z^2,\C)$ is defined as follows. Let $f,g \in\ell^1 (\Z^2,\C)$, then
\begin{equation*}
(f \natural_\theta  g)_{m,n} = \sum_{k,l \in \Z} f_{k,l}g_{m-k,n-l} e^{2\pi i(m-k)l \theta }\,.
\end{equation*}
Moreover, define the involution $f^\ast_{k,l}=\overline{f_{-k,-l}} e^{2\pi ikl\theta }$ for every $f \in \ell^1(\Z^2,\C)$. The triple
$(\ell^1 (\Z^2,\C),\natural_{\theta },^\ast)$ forms a Banach-$^\ast$-algebra --- the so called \textit{twisted convolution algebra}.

The Banach-algebras $\mathcal Q_\theta  \coloneqq \ell^1(\h,\C)/\mathtt J _\theta $ and $(\ell^1 (\Z^2,\C),\natural_{\theta },^\ast)$ are connected by the $^\ast$-isomorphism $\kappa\colon \mathcal Q _\theta \longrightarrow (\ell^1 (\Z^2,\C),\natural_{\theta },^\ast)$ defined by
\begin{equation*}
\kappa (\Phi_\theta  (f))= f^\theta \,.
\end{equation*}

\subsection{Wiener's Lemma for twisted convolution algebras}

Principal results were obtained by Janssen \cite{2Janssen}  and
 Gr\"ochenig and Leinert \cite{2GL}.
Let $\alpha, \beta$ be strictly positive real parameters and let $\theta = \alpha\beta$.
On the Hilbert space $L^2(\R,\C)$ define the translation operator $\textup{T}_\alpha$ and the modulation operator $\textup{M}_\beta$ as follows:
\begin{equation}\label{2def:stonevneu}
(\textup{T}_\alpha F)(t) = F(t+\alpha) \quad \text{and} \quad (\textup{M}_\beta F)(t)= e^{2 \pi i \beta t} F(t)
\end{equation}
where $F\in L^2(\R,\C)$ and $t\in\R$.
The representation $\pi_{\alpha,\beta}$ of  $(\ell^1 (\Z^2,\C),\natural_{\theta},^\ast)$ on $L^2(\R,\C)$ is defined as follows:
for each $f\in\ell^1(\Z^2,\C)$, let
\begin{equation*}
\pi_{\alpha,\beta} (f) =\sum_{k,l\in \Z} f_{k,l} \textup{T}_\alpha ^k \textup{M}_\beta^l \,.
\end{equation*}

Gr\"ochenig and Leinert established the following
Wiener Lemma for twisted convolution algebras.
\begin{theorem}[{\cite[Lemma 3.3]{2GL}}] \label{2t:wt}
Suppose that $\theta \in \T$, $\alpha\beta=\theta \mod 1$, and that
$f\in  \ell^1 (\Z^2,\C)$ and $\pi_{\alpha,\beta}(f)$ is invertible on $L^2(\R,\C)$. Then $f$ is invertible in  $(\ell^1 (\Z^2,\C),\natural_{\theta},^\ast)$.
\end{theorem}

The representation $\pi_{\alpha,\beta}$ of $(\ell^1 (\Z^2,\C),\natural_{\theta},^\ast)$ induces a representation of $\h$, $\ell^1(\h,\C)$ and $\mathcal Q _\theta$ on $L^2(\R,\C)$ in a canonical way:
\begin{equation*}
\pi_{\alpha,\beta} (x)= \textup{T}_\alpha,\quad \pi_{\alpha,\beta} (y)=\textup{M}_\beta , \quad \text{and}\quad \pi_{\alpha,\beta} (z)= e^{2 \pi i \theta} \,.
\end{equation*}
The representations $\pi_{\alpha,\beta}$ appear in the literature under various names: Stone-von Neumann, Weyl-Heisenberg or Schr\"odinger representations.

As an immediate corollary of Theorem \ref{2t:wt} one obtains the following Wiener Lemma for the discrete Heisenberg group.

\begin{theorem}\label{2t:wienerfromconv}
Let $f\in \lh$, then $f$ is invertible if and only if $\pi_{\alpha,\beta}(f)$ is invertible for each non-zero pair $\alpha ,\beta \in \R$.
\end{theorem}

\begin{proof}
The result follows by combining Allan's local principle with Wiener's Lemma for twisted convolution algebras.
\end{proof}

Finally, we give an alternative proof of Wiener's Lemma for the twisted convolution algebra which relies on the representation theory of $\h$ only.
We start with the following lemmas.

\begin{lemma}\label{22lemma:qthetasymmetric}
The twisted convolution algebra $(\ell^1 (\Z^2,\C),\natural_{\theta },^\ast)$ is symmetric.
\end{lemma}
\begin{proof}
First, recall that the Banach algebras $(\ell^1 (\Z^2,\C),\natural_{\theta },^\ast)$ and $\mathcal Q_\theta $ are $^\ast$-isomorphic. For every $f\in\ell^1(\h,\C)$ the following holds: if $\Phi _\theta  (f) \in \mathcal Q _\theta $ is not invertible, then $f$ is not invertible in $\lh$ by Allan's local principle. Hence, $\sigma_{\mathcal Q _\theta }(\Phi _\theta  (f)) \subseteq \sigma_{\lh}(f)$ for every $f\in\ell^1(\h,\C)$. In particular, for every $f\in\lh$,
\begin{equation*}\sigma_{\mathcal Q _\theta }(\Phi _\theta  (f^\ast f)) \subseteq \sigma_{\lh}(f^\ast f)\subseteq [0,\infty)\end{equation*} by the symmetry of $\lh$.
\end{proof}

\begin{lemma}\label{2lemma:convwienerlemmma1}
Consider an irrational $\theta\in \T$. Then $f \in (\ell^1 (\Z^2,\C),\natural_{\theta },^\ast)$ is invertible if and only if
$\pi_{\alpha,\beta}(f)$ has a bounded inverse, where $\alpha\beta=\theta \mod 1$.
\end{lemma}

\begin{proof}
Let $\theta$ be irrational and suppose $a \in \mathcal Q _\theta \simeq (\ell^1 (\Z^2,\C),\natural_{\theta },^\ast)$ is not invertible.
 We just have to show that the non-invertibility of the element $a$ implies that $a$ is not invertible in the irrational rotation algebra $\mathcal R_\theta $ and, in particular,
not in its realisation $\pi_{\alpha,\beta}(C^\ast(\h))$ with $\alpha\beta=\theta \mod 1$. Since  $\mathcal Q_\theta $ is symmetric (cf. Lemma \ref{22lemma:qthetasymmetric}), there exists an irreducible unitary representation $\pi$ of $\h$ such that $\pi$ vanishes on $\mathtt J_\theta $ and $\pi(\Phi_\theta  (a))v =0$ for some non-zero vector $v\in \mathcal H _\pi$. This implies that $a$ is not invertible in $\mathcal R_\theta$.
\end{proof}
The proof of Lemma \ref{2lemma:convwienerlemmma1} basically says that for irrational $\theta$ the Banach algebra $\mathcal Q _\theta$ is inverse-closed in $\mathcal R_\theta$.

We will show that Lemma \ref{2lemma:convwienerlemmma1} holds  for rational $\theta$ as well. 
The representation $\pi_{\alpha,\beta}$ of $\h$ can be decomposed in the following way (cf. \cite{2B}).
Let $\nu$ be the Haar measure on $(0,\theta]$, where $\theta\in \T$ with $\theta = \alpha\beta \mod 1$. There exists a unitary operator
	\begin{equation*}
		 \textup{U} : L^2(\R) \longrightarrow \int_{(0,\theta]}^{\oplus} [\ell^2(\Z,\C)]_t \,d \nu (t)
	\end{equation*}
	and a family of representations $\{ \textup{Ind}^{\h}_N (\sigma_{\theta,s}) \,:\,s\in (0,\theta]\}$ such that
	$\pi_{\alpha,\beta}$ is unitarily equivalent via $ \textup{U}$ to the direct integral
	\begin{equation}\label{2eq:decompofL^2}
	\int_{(0,\theta]}^{\oplus} \textup{Ind}^{\h}_N (\sigma_{\theta,t}) \, d \nu (t)\,.
	\end{equation}
	Since unitary equivalence of two representations implies weak equivalence one gets that
	\begin{equation*}
	\ker (\pi_{\alpha,\beta})=\ker \left(\int_{(0,\theta]}^{\oplus} \textup{Ind}^{\h}_N (\sigma_{\theta,t}) \, d \nu (t)\right)
	= \bigcap_{t\in (0,\theta]} \ker \left(\textup{Ind}^{\h}_N (\sigma_{\theta,t}) \right)=	\mathtt J _{\theta}
	\end{equation*}
	 and hence
	that
	$\pi_{\alpha,\beta}(\ell^1(\h,\C))\simeq \ell^1(\h,\C) / \mathtt J _\theta = \mathcal Q _{\theta}$. From this observation we get the following lemma.
\begin{lemma}\label{2lemma:convwienerlemmma2}
Let $\theta\in\T$ be rational. Then $a\in \mathcal Q _\theta$ is invertible if and only if $\pi_{\alpha,\beta}(a)$ is invertible in
$\mathcal B (L^2(\R,\C))$.
\end{lemma}

\begin{proof}[Proof of Theorem \ref{2t:wt}]
Combine Lemma \ref{2lemma:convwienerlemmma1} and Lemma \ref{2lemma:convwienerlemmma2}.
\end{proof}

\begin{remark}
The decomposition \eqref{2eq:decompofL^2} of $\pi_{\alpha,\beta}$ depends only on the product $\alpha\beta=\theta \mod 1$ and is thus independent of the particular choice of $\alpha$ and $\beta$. Hence, in Theorem \ref{2t:wt} and Theorem \ref{2t:wienerfromconv} one has to consider, e.g., $\alpha=\theta$ and $\beta=1$ only.
\end{remark}

\subsection{An application to algebraic dynamical systems}
As already mentioned in the first section of this article the problem of deciding on the invertibility in $\ell^1(\h,\C)$ has an application in algebraic dynamics.
The following result is important to check invertibility for $f \in \Z[\h]$ in the group algebra $\ell^1(\h,\C)$ because it tells us that
$\pi_{\alpha,\beta}(f)$ has a trivial kernel in $L^2(\R,\C)$ for $\alpha,\beta \not = 0$.
\begin{theorem}[\cite{2L}]
\label{2t:injectivity}
Let $G$ be a non-zero element in $ L^2(\R,\C)$, then for every finite set $A \subseteq \Z^2$ the set $\{\textup{T}_\alpha^k\textup{M}_\beta^l G \,:\, (k,l)\in A\}$
is linear independent over $\C$.
\end{theorem}

The following result is a reformulation of Theorem \ref{2t:injectivity} and gives an exact description of the spectrum of an operator $\pi_{\alpha,\beta}(f)$,
for $\alpha,\beta\in\R\smallsetminus\{0\}$ and $f\in \C[\h]$, where $\C[\h]$ is the ring of functions $\h \longrightarrow \C$ with finite support.
\begin{theorem}
\label{2stonevneumannspectrum}
Let $f \in \C[\h]$ with $f^\theta \not = 0$ for  $\theta=\alpha\beta \not = 0$, $\alpha,\beta \in \R$, then for all $c\in\sigma(\pi_{\alpha,\beta}(f))$
the operators $ \pi_{\alpha,\beta}(c -f)$ are injective and have dense range in $L^2(\R,\C)$ but are not bounded from below.
\end{theorem}

\begin{proof}
Suppose $f\in \C[\h]$ is such that $\pi_{\alpha,\beta}(f) \not = 0$ and $c\in\sigma(\pi_{\alpha,\beta}(f))$. By Theorem \ref{2t:injectivity}, for every non-zero $G\in L^2(\R,\C)$
the finite linear combination
\begin{equation*}
	(c - \pi_{\alpha,\beta}(f))G =	\left( c - \sum_{(k,l,m)\in\Z^3} f_{(k,l,m)}\textup{T}_{\alpha}^k \textup{M}_\beta^l e^{2 \pi i \theta m} \right )G \not = 0 \,.
\end{equation*}
This is equivalent to the injectivity of $c - (\pi_{\alpha,\beta}(f))$.

Suppose that the range of  $c - \pi_{\alpha,\beta}(f)$ is not dense in $L^2(\R,\C)$. Then
\begin{equation*}
(\pi_{\alpha,\beta}(c-f))^\ast=\pi_{\alpha,\beta}((c-f)^\ast)
\end{equation*}
is not injective (cf. the footnote on page \pageref{2footnote:adjoint}) which is a contradiction to Theorem \ref{2t:injectivity} because $(c -f)^\ast \in \C[\h]$.
Hence,  $c - \pi_{\alpha,\beta}(f)$ not being invertible on $L^2(\R,\C)$ is equivalent to $c - \pi_{\alpha,\beta}(f)$ not being  bounded from below.
\end{proof}

Therefore, non-expansiveness of $\alpha_f$ can be checked via two different approaches:
\begin{itemize}
	\item{The dual of $\h$:} there is an irreducible representation $\pi$ of $\h$ such that $0$ is an eigenvalue of $\pi(f)$.
	\item{Stone-von Neumann representations}: For all Stone-von Neumann representations $\pi_{\alpha,\beta}$, $0$ is an eigenvalue of $\pi_{\alpha,\beta}(f)$ if and only if $\pi_{\alpha,\beta}(f)=0$; and
	$\pi_{\alpha,\beta}(f)$ is not invertible if and only if $\pi_{\alpha,\beta}(f)$ is not bounded from below.
\end{itemize}
\begin{remark}
The authors are not aware whether the approach based on Theorem \ref{2t:neumark} and the construction of the dual of $\h$ via ergodic quasi-invariant measures are well-known results in the field of Time-Frequency Analysis. It would be interesting to investigate whether this eigenvalue approach would simplify the problem of deciding on invertibility -- at least -- for some examples  $f \in \ell^1(\h,\C) \smallsetminus  \C[\h]$.
\end{remark}

\section{Examples} \label{2section7}
We now demonstrate how to apply Wiener's Lemma to obtain easily verifiable sufficient conditions for non-expansivity of a principal algebraic action.

Let $f\in \Z[\h]$ be of the form
\begin{equation}\label{2specialform}
f=g_1(y,z)x-g_0(y,z)
\end{equation}
with
$g_1(y,z),g_0(y,z) \in \Z[y,z]\simeq \Z[\Z^2]$.

We set
\begin{equation*}
\mathsf{U}(g_i) = \{(\zeta,\chi)\in\s^2\,:\, g_i(\zeta,\chi)=0\}, \, i=0,1\,.
\end{equation*}

 Moreover, for a polynomial $h\in\C[\Z^d]$ define the
 \emph{logarithmic Mahler measure} $\mathfrak m(h)$
 by the integral
 \begin{equation*}
   \mathfrak m(h) =\int_{\T^d} \log |h(e^{2\pi i\theta_1},
   \ldots, e^{2\pi i\theta_d})| \,
   d\theta_1\cdots d\theta_d\,.
 \end{equation*}

In \cite{2LS2} (cf. \cite[Theorem 2.6]{2GSV} for a proof) the following result was established: for linear $f\in \Z[\h]$ of the form \eqref{2specialform} with $\mathsf{U}(g_i)=\varnothing$ for $i=0,1$, the action
 $\alpha_f$ is expansive  if and only
 if
 \begin{equation*}
 \mathfrak m(g_0)\ne \mathfrak m(g_1)\,.
\end{equation*}

  In this section we use results
 on invertibility to derive criteria for non-expansiveness
 of principal actions of elements $f$ in $\Z[\h]$ of the form
 $f=g_1(y,z)x-g_0(y,z)$  in cases when the unitary
 varieties  $\mathsf{U}(g_0)$ and
$\mathsf{U}(g_1)$ are not necessarily empty.

For every $\chi \in\s$, consider the rational function $\psi_\chi$ on $\s$:
\begin{equation*}
\psi_\chi(\zeta)=\frac{g_0(\zeta,\chi)}{g_1(\zeta\chi^{-1},\chi)}\,
\end{equation*}
and consider the map $\psi \colon \N \times \s \longrightarrow \C$ given by
\begin{equation*}
\psi_{\chi}(n,\zeta)=
\begin{cases} 1 &\mbox{if } n = 0 \\
 \prod_{j=0}^{n-1}\psi_\chi(\zeta\chi^{-j}) & \mbox{if } n\geq 1.
 \end{cases}
\end{equation*}

\subsection{Either $\mathsf{U}(g_0)$ or $\mathsf{U}(g_1)$ is a non-empty set}

We fix the following notation. For every $\chi\in\s$ and $i=0,1$, put
\begin{equation*}
\mathsf{U}_\chi (g_i)=\{\zeta\in \s\,:\,g_i(\zeta,\chi)=0\}\,,
\end{equation*}
and
\begin{equation*}
g_{i,\chi}(y) =g_i(y,\chi)\,,
\end{equation*}
which we will view as a Laurent polynomial in $y$
with complex coefficients, i.e., $g_{i,\chi}\in \C[\Z]$
for every $\chi$ and $i=0,1$. Note also that
the set $\mathsf{U}_\chi (g_i)$ is
infinite   if and only if  $g_{i,\chi}$
is the zero polynomial.

For notational convenience we put
\begin{equation*}
\phi_\chi(\zeta)=\log |\psi_\chi(\zeta) | \quad \text{and}\quad \phi_\chi(n,\zeta)=\log |\psi_\chi (n,\zeta) |\,,
\end{equation*}
for every $\zeta \in \s$ and $n\geq0$.

\begin{theorem}\label{2t:oneemptyonenot}
Let $f\in\Z[\h]$ be of the form $f=g_1(y,z)x-g_0(y,z)$. Suppose there exists an element $\chi\in\s$ of infinite order which satisfies either of the following conditions.
\begin{enumerate}[label=(\roman{*})]
\item \label{2cond:1}
$\mathsf{U}_\chi(g_0)=\varnothing$, $\mathsf{U}_\chi(g_1)\not=\varnothing$ and $\int \phi_\chi d\lambda_{\s} <0$.
\item \label{2cond:12}
$\mathsf{U}_\chi(g_0)\not=\varnothing$, $\mathsf{U}_\chi(g_1)=\varnothing$ and $\int \phi_\chi d\lambda_{\s} >0$.
\end{enumerate}
Then $\alpha_f$ is non-expansive.
\end{theorem}

\begin{proof}
We will prove only the first case,
the second case can be proved similarly.

Suppose $f$ is such that $(X_f,\alpha_f)$ is expansive
and the conditions in \ref{2cond:1} are satisfied. We will now show that
certain consequences of expansivity of $\alpha_f$
are inconsistent with the conditions in
\ref{2cond:1}. Hence, by arriving to a contradiction,
we will prove that under \ref{2cond:1} $\alpha_f$ is not expansive.

We know that $(X_f,\alpha_f)$ is expansive if and only if $f$ is invertible in $\lh$.
Hence $(X_f,\alpha_f)$ is expansive if and only if there exists a $w\in\lh$,
\begin{equation*}
 w = \sum_{k,l,m} w_{k,l,m} y^l x^k  z^m\,,
\end{equation*}

such that
\begin{equation*}
f\cdot w = w\cdot f = 1_{\lh}\,.
\end{equation*}

Suppose  $\theta\in (0,1]$ is irrational and that
$\chi=e^{2\pi i\theta}\in\mathbb S$ satisfies condition
\ref{2cond:1}.
Consider the following representation $\pi_{1,\theta}$
of $\ell^1(\h,\C)$ on $L^2(\mathbb R,\C)$, defined by
\begin{align*}
(\pi_{1,\theta}(x) F)(t)= \textup{T}_1F(t) = F(t+1)&,
\quad (\pi_{1,\theta}(y) F)(t)  = \textup{M}_\theta F(t)=
e^{2\pi i\theta t} F(t),
\\\text{and}\quad &\ (\pi_{1,\theta}(z) F)(t) = e^{2\pi i\theta} F(t)\,.
\end{align*}

If
\begin{equation*}
f= g_1(y,z)x-g_0(y,z) \quad \text{and} \quad w=f^{-1} =
\sum_{k,l,m} w_{k,l,m} y^k x^l  z^m\,,
\end{equation*}
then
\begin{equation*}
\pi_{1,\theta}(f) =  g_1(e^{2\pi i\theta t},e^{2\pi i\theta})\textup{T}_1
-g_0(e^{2\pi i\theta t},e^{2\pi i\theta})
\end{equation*}
and
\begin{equation*}\aligned
\pi_{1,\theta}(w) &= \sum_{(k,l,m)\in\Z^3}
w_{k,l,m} \textup{M}^k_\theta \textup{T}_1^l  {\chi}^m
\\&=
\sum_{l\in \Z}
\Bigl[ \sum_{(k,m)\in\Z^2}
w_{(k,l,m)}e^{2\pi i\theta t k} e^{2\pi i\theta m}
\Bigr] \textup{T}_1^l\,.
\endaligned
\end{equation*}

Set
\begin{equation*} P_{l,\theta}(t)\coloneqq \sum_{(k,m)\in\Z^2}
w_{(k,l,m)}e^{2\pi i\theta t k} e^{2\pi i\theta m}\,,
\end{equation*} then
$\pi_{1,\theta}(w)=\sum_{l\in \Z} P_{l,\theta}(t) \textup{T}_1^l$.

The functions $P_{l,\theta}(\cdot)\colon \R\longrightarrow \C$,
$l\in\Z$, are bounded and continuous.  Indeed,
for any $l\in \Z$
\begin{equation*}
P_{l,\theta}(t)=\sum_{(k,m)\in\Z^2}
w_{(k,l,m)}e^{2\pi i\theta t k} e^{2\pi i\theta m}
\end{equation*}
is a Fourier series with absolutely convergent
coefficients:
\begin{align*}
\sum_{k\in \Z} \Bigl|\sum_{m\in\Z}
w_{(k,l,m)}  e^{2\pi i\theta m}\Bigr|
&\le
 \sum_{k\in \Z}  \sum_{m\in\Z} |
w_{(k,l,m)}  | \\&\le \|w\|_{\ell^1(\h,\C)}<\infty\,.
\end{align*}
For similar reasons,
\begin{align}\label{2eq:norm1}
 \sum_{l\in \Z} \sup_{t\in\R}| P_{l,\theta}(t)| &\le
 \sum_{l\in \Z} \Bigl[ \sum_{k,m} |w_{k,l,m}|\Bigr]
 \\ &=\|w\|_{\ell^1(\h,\C)}\label{2eq:norm2}
 <\infty\,.
 \end{align}

 Since $w\cdot f = 1_{\lh}$ and $\pi_{1,\theta}(1_{\lh}) = 1_{\mathcal B (L^2(\R,\C))}$ -- the identity operator on $L^2(\R,\C)$, one has
 \begin{equation*}\aligned
 1_{\mathcal B (L^2(\R,\C))} &= \pi_{1,\theta}(w)\pi_{1,\theta}(f)\\
 &= \Bigl[ \sum_{l\in\Z} P_{l,\theta}(t) \textup{T}_1^l\Bigr]
 \cdot\Bigl[ g_1(e^{2\pi i \theta t},e^{2\pi i \theta})\textup{T}_1
 - g_0(e^{2\pi i \theta t},e^{2\pi i \theta})\Bigr] \\
 &=\Bigl[ \sum_{l\in\Z} P_{l,\theta}(t) \textup{T}_1^l\Bigr]
 \cdot\Bigl[ g_{1,\chi}(e^{2\pi i \theta t})\textup{T}_1
 - g_{0,\chi}(e^{2\pi i \theta t})\Bigr] \\
 &= \sum_{l\in\Z} \Bigl[  P_{l-1,\theta}(t)
 g_{1,\chi}(e^{2\pi i\theta(t+l-1)})-
  P_{l,\theta}(t)
 g_{0,\chi}(e^{2\pi i\theta(t+l)})
 \Bigr] \textup{T}_1^l \,.
 \endaligned
 \end{equation*}

 Set
\begin{equation}
 Q_{l,\theta}(t)=P_{l-1,\theta}(t)
 g_{1,\chi}(e^{2\pi i\theta(t+l-1)})-
  P_{l,\theta}(t)
 g_{0,\chi}(e^{2\pi i\theta(t+l)})\,.
 \end{equation}
 Since $\{Q_{l,\theta}(\cdot)|\, l\in\Z\}$ are again bounded continuous functions, one
 concludes that
 \begin{equation*}
 Q_{0,\theta}(t)\equiv 1 \quad \text{and}\quad Q_{l,\theta}(t)\equiv 0,\quad \text{for every $l\ne 0$.}
 \end{equation*}
 Hence, for every $t\in\mathbb R$, one has
\begin{equation}\label{2p_zero}
 Q_{0,\theta}(t)=
 P_{-1,\theta}(t)
 g_{1,\chi}(e^{2\pi i\theta(t-1)})-
  P_{0,\theta}(t)
 g_{0,\chi}(e^{2\pi i\theta t})
=1\,,
\end{equation}
and for every $l\ge 1$
\begin{equation}\label{2p_l}
Q_{l,\theta}(t) = P_{l-1,\theta}(t)
 g_{1,\chi}(e^{2\pi i\theta(t+l-1)})-
  P_{l,\theta}(t)
 g_{0,\chi}(e^{2\pi i\theta(t+l)})=0\,.
\end{equation}
Since $\mathsf{U}_\chi(g_0)=\varnothing$,
equations \eqref{2p_l} imply that
\begin{equation*}
P_{l,\theta}(t)=P_{l-1,\theta}(t)\cdot \frac
{g_{1,\chi}(e^{2\pi i\theta(t+l-1)})}{g_{0,\chi}(e^{2\pi i\theta(t+l)})}
\end{equation*}
for every $t$, and hence for each $l\ge 1$ and every $t\in\R$, one has
\begin{equation*}
P_{l,\theta}(t) = P_{0,\theta}(t) \cdot \frac
{g_{1,\chi}(e^{2\pi i\theta t})}{g_{0,\chi}(e^{2\pi i\theta(t+1)})}\cdots \frac
{g_{1,\chi}(e^{2\pi i\theta(t+l-1)})}{g_{0,\chi}(e^{2\pi i\theta(t+l)})}\,,
\end{equation*}
or
\begin{equation*}
P_{l,\theta}(t) = P_{0,\theta}(t) \frac{1}{\psi_{\chi^{-1}}(l,\zeta_t \chi)}\,
\end{equation*}
where $\zeta_t = e^{2\pi i\theta t}$.

 Then since
$\mathsf{U}_\chi(g_0)=\varnothing$, the logarithmic Mahler
measure $\mathfrak m(g_{0,\chi})$ is finite, and hence
$\mathfrak m(g_{1,\chi})>-\infty$. Therefore, $g_{1,\chi}(\eta)$
is not identically $0$ on $\mathbb S^1$,
and since $g_{1,\chi}$ is a polynomial, we
can conclude that  $\mathsf{U}_\chi(g_1)$ is finite.
Therefore, the set of points
\begin{equation*}
 B_1=\Bigl\{ \zeta \in\mathbb S^1:  \zeta e^{2\pi i \theta k}\in
\mathsf{U}_\chi(g_1)\quad\text{for some }k\in\Z\Bigr\}
=\bigcup_{k\in\Z} \textup{R}_\theta^k( \mathsf{U}_\chi(g_1))
\end{equation*}
is at most countable, and hence has Lebesgue measure $0$.

Both functions $\log|g_{0,\chi}(\cdot)|$
and $\log|g_{1,\chi}(\cdot)|$ are integrable. Moreover, the irrational rotation $\textup{R}_\theta\colon \T \longrightarrow \T$ is an ergodic transformation.
 By Birkhoff's
ergodic theorem there exists a set $B_2\subset\mathbb S^1$
of full Lebesgue measure such that for any $\zeta\in B_2$
\begin{equation*}
\aligned
\frac 1n \sum_{k=1}^{n} \log &|g_{0,\chi}(\zeta e^{2\pi i\theta k})|\to \mathfrak m( g_{0,\chi})\,,\\
\frac 1n \sum_{k=0}^{n-1} \log &|g_{1,\chi}(\zeta e^{2\pi i\theta k})|\to \mathfrak m( g_{1,\chi})\,.
\endaligned
\end{equation*}
Therefore, since $\mathfrak m( g_{1,\chi})>\mathfrak m( g_{0,\chi})$, on the set of full measure $B_1^c\cap B_2$
\begin{equation}
\label{2eq:Psi}
   \Psi_n(\zeta) = \frac{1}{\psi_{\chi^{-1}}(n,\zeta \chi)}
\ne 0 \quad \forall n\ge 1,
 \end{equation}
and
\begin{equation}\label{2eq:PsiLim}
\lim_{n\to\infty} \Psi_n(\zeta)= +\infty\,.
\end{equation}
Since $\theta\ne 0$,  the set of points
\begin{equation*}
 C=\Bigl\{t\in\mathbb R\,:\,  e^{2\pi i t\theta} \not\in B_1^c\cap B_2\Bigr\}
\end{equation*}
has full measure, and for every $t\in  C$,
one has that the sum
\begin{equation}
\label{2eq:P01}
\sum_{l\ge 0} |P_{l,\theta}(t)| =|P_{0,\theta}(t)|
+|P_{0,\theta}(t)|\Psi_1(e^{2\pi i t\theta})+
\ldots+
|P_{0,\theta}(t)|\Psi_l(e^{2\pi i t\theta})+\ldots
\end{equation}
is finite if and only if $|P_{0,\theta}(t)|=0$.
Combining this fact with the uniform bound \eqref{2eq:norm1} -- \eqref{2eq:norm2},
we are able to conclude that
\begin{equation}
\label{2eq:P02}
P_{0,\theta}(t) =0
\end{equation}
on a set of full measure in $\R$. The function $P_{0,\theta}(t)$ is continuous and therefore, $P_{0,\theta}$ must be
the identically zero function on $\R$.

Finally, consider the remaining equation \eqref{2p_zero}
 for $Q_{0,\theta}(t)$. Since $P_{0,\theta}(t)\equiv 0$,
one has that there exists a continuous bounded
function $P_{-1,\theta}$ such that
\begin{equation} \label{2eq:P-1}
P_{-1,\theta}(t)
 g_{1,\chi}(e^{2\pi i\theta(t-1)})
= 1 \,,
\end{equation}
for every $t\in \R$. However, since the unitary variety
$\mathsf{U}_{\chi}(g_1)$ is not empty, one can find $t\in \R$
such that
\begin{equation*}
g_{1,\chi}(e^{2\pi i\theta(t-1)})=0\,,
\end{equation*}
and hence, \eqref{2eq:P-1} cannot be satisfied.
 Therefore, we arrived to a contradiction with
the earlier assumption that $\alpha_f$ is expansive.
\end{proof}

The assumption that $\int \phi_\chi d\lambda_{\s} < 0$ cannot be dropped in \ref{2cond:1} of Theorem \ref{2t:oneemptyonenot} as the following simple minded example shows.
\begin{example}
Suppose $\chi$ is not a root of unity  and $\mathsf{U}_\chi(g_1)\not=\varnothing$. Set $g_0(y,z)\equiv K$, where we pick $K \in \N$ such that
\begin{enumerate}
\item  $K> \|g_1(y,z)\|_{\ell^1(\h,\C)}$;
\item $\int \phi_\chi d\lambda_{\s} > 0$.
\end{enumerate}
Then $f=g_1(y,z)x-K$ is invertible since
\begin{equation*}
f= K \left ( \frac{g_1(y,z)x}{K} - 1 \right )\quad \text{and}\quad \left\| \frac{g_1(y,z)x}{K}\right \|_{\ell^1(\h,\C)} < 1 \,.
\end{equation*}
\end{example}

\subsection{The sets $\mathsf{U}(g_0)$ and $\mathsf{U}(g_1)$ are both non-empty}
Let us denote by
\begin{equation*}
\textup{Orb}_{\chi}(\zeta) =\{\zeta\chi^{n} \,:\, n\in\Z \}
\end{equation*} the orbit of $\zeta$ under the circle rotation $\textup{R}_\chi\colon \s\longrightarrow \s$ with $ \zeta \mapsto \zeta\chi$, for every $\zeta\in\s$.
We consider first linear elements $f=g_1(y,z) x-g_0(y,z)$ for which
\begin{enumerate}
	\item $\mathsf{U}(g_0) \not = \varnothing$ and $\mathsf{U}(g_1) \not = \varnothing$;
	\item and there exists an element $(\zeta,\chi) \in \mathsf{U}(g_0)$ with
	\begin{equation*}
	\{ (\eta,\chi)\in \s^2\,:\, \eta\in \textup{Orb}_{\chi}(\zeta)\} \cap \mathsf{U}(g_1) \not = \varnothing\,.
	\end{equation*}
\end{enumerate}

\begin{theorem}\label{2t:linear}
If $f$ is of the form \eqref{2specialform} and for some $m\in \Z$
\begin{equation*}
\mathsf{U}(g_1) \cap \left\{( \xi\chi^m ,\chi)\in\s^2\,:\, (\xi,\chi)\in\mathsf{U}(g_0)\right\} \not = \varnothing\,,
\end{equation*}
then $\alpha_f$ is non-expansive.
\end{theorem}

Although this result could be proved with the help of Stone-von Neumann representations as well, it is more suitable to use monomial representations.
 For the following discussion it is convenient to work with slightly modified versions of the monomial representations defined in
\eqref{2e:inducedstonevonneumann}.
For every $\zeta,\chi \in \s$ let $\pi^{(\zeta,\chi)}$ be the representation of $\h$ acting on $\ell^2(\Z,\C)$ which fulfils
\begin{align}\label{2e:repzetachi}
(\pi^{(\zeta,\chi)}(x)F)(n)=F(n+1)&, \quad (\pi^{(\zeta,\chi)}(y)F)(n)=\zeta \chi^n F(n) \quad \text{and}
\\ &(\pi^{(\zeta,\chi)}(z)F)(n)=\chi F(n)
\end{align}
for each $F\in \ell^2(\Z,\C)$ and $n\in\Z$.

\begin{proof}
Consider the case $m \geq 0$ first.
Suppose $f$ is invertible and hence $\pi(f)$ is invertible for every unitary representations of $\h$ and in particular, $\pi^{(\zeta,\chi)}(f)$ is invertible for every pair $(\zeta,\chi) \in \s^2$.
By the assumptions of the theorem there exists a pair $(\xi,\chi) \in \s^2$ such that
\begin{equation}\label{2eq:choice}
g_1( \xi\chi^m,\chi)=0 \quad\text{and}\quad g_0(\xi,\chi)=0\,.
\end{equation}
Without loss of generality we may assume that
$m$ is the minimal power such that (\ref{2eq:choice})
is satisfied, i.e.,  $g_1( \xi\chi^l,\chi)\not=0$ for each $l\in\Z$ with $0\leq l\leq m-1$.

Suppose $G\in \ell^2(\Z,\C)$ is in the image of $\pi^{(\xi,\chi)}(f)$, then there exists an $F\in \ell^2(\Z,\C)$ such that
\begin{equation}\label{2e:rep}
G(n)=g_1(\xi\chi^n,\chi)F(n+1)-g_0(\xi\chi^n,\chi)F(n)
\end{equation}
holds for every $n\in\Z$. Given the choice of $(\xi,\chi)$
(cf. (\ref{2eq:choice})), one immediately concludes that
\begin{equation*}
G(0)=g_1(\xi,\chi)F(1)\,.
\end{equation*}

If $m=0$, then $G(0)=0$, and we arrive to a contradiction with the assumption that $\pi^{(\xi,\chi)}(f)$ is invertible, and hence has a dense range in
$\ell^2(\Z,\C)$: Indeed, for every $F\in \ell^2(\Z,\C)$ one has that
\begin{equation*}
(\pi^{(\xi,\chi)}(f)F)(0)=0\,
\end{equation*}
and hence, the range of $\pi^{(\xi,\chi)}(f)$ is not
dense.

If $m >0$, then $F$ must satisfy the following system of linear equations
\begin{align} \label{2eq:systemoflinearequtions1}
G(0)&=g_1(\xi,\chi)F(1) \\
\label{2eq:systemoflinearequtions2}
G(l)&=g_1(\xi\chi^l,\chi)F(l+1)-g_0(\xi\chi^l,\chi)F(l), \quad 1 \leq l \leq m-1 \\
\label{2eq:systemoflinearequtions3}
G(m)&=-g_0(\xi\chi^m,\chi)F(m)\,.
\end{align}

We can eliminate $F(1),F(2),\ldots ,F(m)$ in \eqref{2eq:systemoflinearequtions1} -- \eqref{2eq:systemoflinearequtions3} to obtain an expression for $G(m)$ in terms of
$G(0),G(1),\ldots,G(m-1)$. Indeed,
one can easily verify that, for each $0\leq l \leq m-1$, $F(l+1)$ can be written as
\begin{align}
\label{2Fs+l}
F(l+1)&=\frac {G(l)}{g_1(\xi\chi^l,\chi)}+ \frac { g_0(\xi\chi^l,\chi)}{g_1(\xi\chi^l,\chi)}F(l)=\ldots\\
&= \sum_{i=0}^l\frac{G(l-i)}{g_1(\xi\chi^{l},\chi)}\prod_{n=0}^{i-1}\frac{g_0(\xi\chi^{l-n},\chi)}{g_1(\xi\chi^{l-n-1},\chi)}
\\
&=\sum_{i=0}^l\frac{G(l-i)}{g_1(\xi\chi^{l},\chi)} \psi_\chi(i,\xi\chi^l)
\end{align}
(we use the convention that the empty product $\prod_{\varnothing}$  is equal to $1$).
Due to our choice of $m$, $F(l+1)$ in \eqref{2Fs+l}  is  well-defined.
Moreover, since $G(m)=g_0(\xi\chi^m,\chi)F(m)$, one gets
\begin{align*}
G(m)&=-g_0(\xi\chi^{m},\chi) F(m)\\&=
 -\sum_{i=0}^{m-1}G(m-1-i) \psi_\chi(i+1,\xi\chi^{m-1})\,.
\end{align*}

Hence, $G(m)$ depends continuously on the values $G(0),G(1),\ldots,G(m-1)$.
Again, this is a contradiction  with our
hypothesis that $\pi^{(\xi,\chi)}(f)$ has dense range in $\ell^2(\Z,\C)$.

If $m<0$, then we choose $\pi$ such that
\begin{equation*}
(\pi(x)F)(n)=F(n-1) \quad \text{and} \quad(\pi(y)F)(n)=\xi \chi^{-n} F(n)
\end{equation*}
for each $F\in \ell^2(\Z,\C)$ and $n\in\Z$.
Exactly the same arguments as in the case $m \geq 0$ can be used to get a contradiction to the invertibility of $\pi(f)$.
\end{proof}

\begin{corollary}\label{2coro:f'}
Let $f$ be of the form \eqref{2specialform} which satisfies the conditions of Theorem \ref{2t:linear}. Then $\alpha_{f^\diamond }$ defined by $f^\diamond =g_0(y,z)x-g_1(y,z)$ is non-expansive.
\end{corollary}
\begin{proof}
Since
\begin{equation*}
\mathsf{U}(g_1) \cap \left\{( \xi\chi^m ,\chi)\in\s^2\,:\, (\xi,\chi)\in\mathsf{U}(g_0)\right\} \not = \varnothing\,,
\end{equation*}
there exists a $k\in\Z$ such that
\begin{equation*}
\mathsf{U}(g_0) \cap \left\{( \xi\chi^k ,\chi)\in\s^2\,:\, (\xi,\chi)\in\mathsf{U}(g_1)\right\}
\end{equation*}
 is non-empty as well. Hence, Theorem \ref{2t:linear} guarantees the non-expansiveness of $\alpha_{f^\diamond}$.
\end{proof}

\begin{example}
Consider
\begin{equation*}g_1(y,z)=1-y-y^{-1}-z-z^{-1} \quad \text{and} \quad g_0(y,z)=3-y-y^{-1}-z-z^{-1}\,.\end{equation*}
We will show that the dynamical systems $(X_f,\alpha_f)$ and $(X_{f^\diamond},\alpha_{f^\diamond})$, with $f=g_1(y,z)x -g_0(y,z)$ and $f^\diamond =g_0(y,z)x -g_1(y,z)$,
are non-expansive.

For this purpose we introduce, for $m\in \Z$, the `$m$-sheared version' of $g_0(y,z)$ given by
\begin{equation*}
g_0^{(m)}(y,z)=g_0(yz^m,z)= 3- y z^m-y^{-1} z^{-m}- z-z^{-1}\,;
\end{equation*}
and note that
\begin{equation*}
\mathsf{U}(g_1) \cap \left\{( \xi\chi^m ,\chi)\in\s^2\,:\, (\xi,\chi)\in\mathsf{U}(g_0)\right\} \not = \varnothing \iff \mathsf{U}(g_1) \cap \mathsf{U}(g_0^{(m)})\not = \varnothing\,.
\end{equation*}

The Fourier transforms of $g_1(y,z)$ and $g^{(m)}_0(y,z)$ are given by the functions
\begin{align*}
(\mathcal F g_1)(s,t)&= 1-2 \cos(2 \pi  s) - 2 \cos(2 \pi  t) \quad \text{and}\\ \left(\mathcal F g_0^{(m)} \right) (s,t)&= 3- 2 \cos(2 \pi  (s+mt))-2 \cos(2 \pi  t) \,,
\end{align*}
respectively.

\begin{figure}[tbph]
\begin{center}
\includegraphics[width=0.70\textwidth]{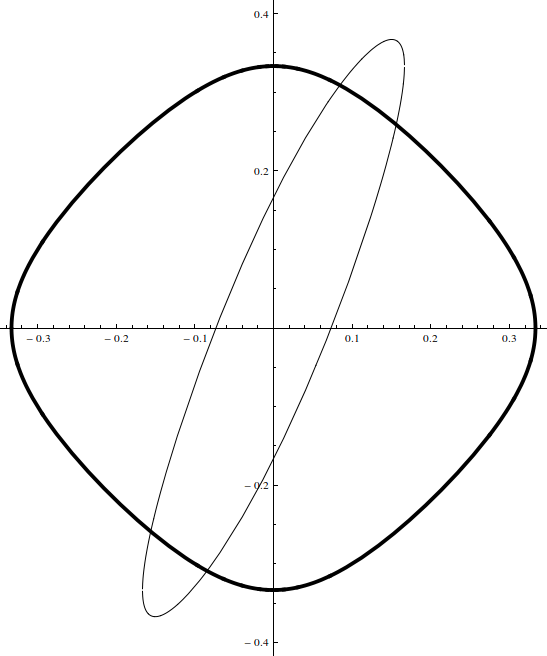}
\end{center}

\caption[.] {In this Figure the curves corresponding to the solution sets $K$ (thick line) and $K[2]$ (thin line) are plotted.}
\label{2f:1}
\end{figure}
Let
\begin{align*}
K &= \{(s,t)\in \T^2 \,:\, (\mathcal F g_1)(s,t)=0\} \quad \text{and}\\ K[m]&=\left\{(s,t)\in \T^2 \,:\, \left(\mathcal F g_0^{(m)} \right)(s,t)=0\right\}\,.
\end{align*}
Fix $m\in \Z$. By solving the equations
\begin{equation*}
	(\mathcal F g_1)(s,t)=0 \quad \text{and} \quad  \left(\mathcal F g_0^{(m)} \right)(s',t')=0
\end{equation*}
for $s$ and $s'$ we get curves $s(t)$ and $s'(t')$ corresponding to the solution sets $K$ and $K[m]$.
If these curves intersect, then $K$ and $K[m]$ have a non-empty intersection.
It is clear that $(s,t) \in K$ if and only if $(e^{2 \pi i s},e^{2 \pi i t})\in\mathsf{U}(g_1)$. For every $m\in\Z$ the sets $K[m]$ and $\mathsf{U}(g_0^{(m)})$ are related in the same way.
The sets $K$ and $K[2]$ have a non-empty intersection as  Figure \ref{2f:1} shows; while $K \cap K[0] = \varnothing$  and $K \cap K[1] = \varnothing$.

Since the conditions of Theorem \ref{2t:linear} and Corollary \ref{2coro:f'} are satisfied,
$f$ and $f^\diamond$ are not invertible.

\end{example}

The next result can be  easily deduced from the proof of Theorem \ref{2t:oneemptyonenot}.

\begin{theorem}
Let $f\in\Z[\h]$ be of the form \eqref{2specialform}. Suppose there exists an element $\chi\in\s$ of infinite order such that the following conditions are satisfied
\begin{equation} \label{2e:cond:1}
\mathsf{U}_\chi(g_0)\not =\varnothing \quad \text{and} \quad \mathsf{U}_\chi(g_1)\not=\varnothing \,,
\end{equation}
and
\begin{equation*}
\mathfrak{m} (g_{0,\chi}) \not = \mathfrak{m} (g_{1,\chi})\,.
\end{equation*}
Then $\alpha_f$ is non-expansive.
\end{theorem}
\begin{proof}
Suppose \eqref{2e:cond:1} holds.
Let us first treat the trivial cases.

If $g_{0,\chi}(y)$ is the zero-element in $\C [\Z]$, then for every $\zeta\in \s$
\begin{equation*}
	\pi^{(\zeta,\chi)}(f)= \pi^{(\zeta,\chi)}(g_1(y,z)x)\,.
\end{equation*}
Fix $\xi \in \mathsf{U}_\chi(g_1)$, which is a non-empty set by the assumptions of the theorem.
Since one has $(\pi^{(\xi,\chi)}(f)F)(0)$ is equal to $0$ for every $F\in \ell^2(\Z,\C)$, $0$ is an element of $\sigma (\pi^{(\xi,\chi)}(f))$ and hence $f$ is not  invertible. The same conclusions can be drawn for the cases $g_{1,\chi}= 0_{\C[\Z]}$ and $g_{0,\chi}= g_{1,\chi}= 0_{\C[\Z]}$.

Next consider the case where $g_{0,\chi}$ and $g_{1,\chi}$ are not the zero elements in $\C[\Z]$, which implies that $\mathfrak{m} (g_{0,\chi})$  and  $\mathfrak{m} (g_{1,\chi})$ are finite and moreover $\mathsf{U}_\chi(g_0)$ and $\mathsf{U}_\chi(g_1)$ are finite sets.
Suppose that $\mathfrak{m} (g_{0,\chi})< \mathfrak{m} (g_{1,\chi})$. We follow the line of arguments in the proof of Theorem \ref{2t:oneemptyonenot}.
The only adaption one has to make is to take  the countable set
\begin{equation*}
B=\Bigl\{ t\in\R\,:\,  e^{2 \pi i t} \chi^{k}\in
\mathsf{U}_\chi(g_0)\,\,\text{for some }k\in\Z\Bigr\}
\end{equation*}
into consideration, i.e., to exclude points in $\mathcal B$ in the equations \eqref{2eq:Psi} -- \eqref{2eq:P02}.

The case $\mathfrak{m} (g_{0,\chi}) > \mathfrak{m} (g_{1,\chi})$ can be proved analogously.
\end{proof}



\end{document}